\documentclass[10pt]{amsart}

\usepackage[english]{babel}
\usepackage[T1]{fontenc}
\usepackage[utf8]{inputenc}
\usepackage{amsmath}
\usepackage{amssymb}
\usepackage{amsthm}
\usepackage{color}
\usepackage[a4paper]{geometry}
\usepackage{tikz}
\usepackage{mathdots}
\usepackage[all]{xy}
\usepackage{hyperref}
\usepackage{accents}
\usepackage{lmodern}
\usepackage{mathrsfs}

\theoremstyle{definition}
\newtheorem{definition}{Definition}[section]

\theoremstyle{plain}
\newtheorem{theorem}[definition]{Theorem}
\newtheorem{prop}[definition]{Proposition}

\newtheorem{lem}[definition]{Lemma}

\newcommand{\mb}{\mathbb}
\newcommand{\mc}{\mathcal}
\newcommand{\mf}{\mathfrak}

\newcommand{\bs}{\boldsymbol}

\newcommand{\Det}{\textup{Det}}

\begin{document}

\title{Sums of reciprocals of fractional parts over aligned boxes}

\author{Reynold Fregoli}
\address{Department of Mathematics\\ 
University of York\\ 
Heslington, York, YO10 5DD\\ 
UK}
\email{reynold.fregoli@york.ac.uk}

\subjclass{Primary ; Secondary }
\date{\today, and in revised form ....}

\dedicatory{}


\begin{abstract}
In this paper, we prove new upper bounds for sums of reciprocals of fractional parts over general aligned boxes, thus extending a previous result of the author concerning bounds for sums of reciprocals over symmetric boxes. These new upper bounds depend solely on the volume of the boxes, and not on their diameter. This generalisation relies a novel lattice-point counting technique involving estimates for the higher successive minima of certain naturally arising lattices.
\end{abstract}

\maketitle

\section{Introduction}

\subsection{Notation}

We denote by $|\bs{v}|_{2}$ the Euclidean norm of a vector $\bs{v}\in\mb{R}^{n}$, and we denote by $|\bs{v}|_{\infty}$ its maximum norm. We denote by $\|x\|$ the distance from $x\in\mb{R}$ to the nearest integer, and by $\mb{N}$ the set $\{1,2,3\dotsc\}$ of positive integers. For any set $X$ and any pair of functions $f,g:X\to\mb{R}$ we write $f\ll g$ ($f\gg g$) when there exists a real constant $c>0$ such that $f(x)\leq cg(x)$ ($f(x)\geq cg(x)$) for all $x\in X$. If the constant $c>0$ depends on any parameters, we indicate these next to the symbol $\ll$ ($\gg$). For $i=1,\dotsc,n$ the phrase "$i$-th successive minimum" referred to a lattice $\Lambda\subset\mb{R}^{n}$ will indicate the infimum of all real numbers $r>0$ such that the set $\Lambda\cap B(\bs{0},r)\subset\mb{R}^{n}$ contains at least $r$ linearly independent distinct vectors, where $B(\bs{0},r)\subset\mb{R}^{n}$ stands for the ball of radius $r$ and centre located at the origin.

\subsection{Background and Main Results}

In this manuscript, we prove new general upper bounds for sums of reciprocals of fractional parts, which fully complement analogous lower bounds obtained by L\^e and Vaaler in \cite{LeVaaler:Sumsof}. In \cite{Fregoli:Onacounting}, we had proved similar upper bounds under additional symmetry hypotheses dictated by the particular techniques used to approach the problem, which we now remove. These techniques involved a generalised notion of weak admissibility for lattices \cite{Widmer:WeakAdmiss}, and a powerful lattice-point counting estimate for fibres of o-minimal structures \cite{Widmer:CountingLattice}. In the more general case treated here, however, the potential distortion of the domain over which the sums in question are defined requires a novel and more careful argument. This argument consists of two main components: a sophisticated partition method developed by Widmer and the author in \cite{Widmer:AsymptoticDiophantine} and \cite{Fregoli:Onacounting}, and a novel lattice-point counting technique based on estimates for the higher successive minima of certain lattices, which constitutes the main contribution of this paper. Before presenting the main result and its proof, we proceed to recall a few key facts about sums of reciprocals of fractional parts and estimates thereof.

Let $\mc{M},\mc{N}\in\mb{N}$ and let $\bs{L}\in\mb{R}^{\mc{M}\times \mc{N}}$ be a matrix with $\mc{M}$ rows and $\mc{N}$ columns. We indicate by $L_{1},\dotsc,L_{\mc{M}}$ the rows of the matrix $\bs{L}$ and we denote by 
$$L_{i}\bs{x}:=\sum_{j=1}^{\mc{N}}L_{ij}x_{j}$$
the product of the $i$-th row of $\bs{L}$ by the vector $\bs{x}\in\mb{R}^{\mc{N}}$. Throughout, as a technical assumption, we require that the entries $L_{i1},\dotsc,L_{i\mc{N}}$ of of the $i$-th row of the matrix $\bs{L}$ along with the integer $1$ are linearly independent over $\mb{Q}$ for $i=1,\dotsc,\mc{M}$.

Let $\bs{L}\in\mb{R}^{\mc{M}\times \mc{N}}$ be a matrix as above, and let $\bs{Q}:=(Q_{1},\dotsc,Q_{\mc{N}})\in[1,+\infty)^{\mc{N}}$. The main object discussed in this paper is the following family of functions depending on the parameter $\bs{L}$:
\begin{equation}
\label{eq:intro1}
S_{\bs{L}}(\bs{Q}):=\sum_{\bs{q}\in\prod_{j=1}^{\mc{N}}[-Q_{j},Q_{j}]\cap\in\mb{Z}^{\mc{N}}\setminus\{\bs{0}\}}\prod_{i=1}^{\mc{M}} \|L_{i}\bs{q}\|^{-1}.
\end{equation}
\noindent Sums of this shape are commonly known as sums of reciprocals of fractional parts and have been widely studied in light of their numerous applications, spanning, e.g., the Geometry of Numbers \cite{HarLi:SomeProblems2}, the Theory of Uniform Distribution \cite{KuipersNiederr:UnifDistr}, Diophantine Approximation \cite{HuangLiu:Simultaneous}, and many other strands of Mathematics. We refer the reader to \cite{Beresnevich:Sumsofandmult},\cite[Applications I]{Fregoli:Onacounting},\cite{Fregoli:Sumsof}, and \cite{LeVaaler:Sumsof} for a deeper insight.

In \cite{LeVaaler:Sumsof}, L\^e and Vaaler obtained general lower bounds for the functions $S_{\bs{L}}(\bs{Q})$ in terms of the quantity $Q:=(Q_{1}\dotsc Q_{\mc{N}})^{1/\mc{N}}$, i.e., the geometric mean of the variables\footnote{In what follows, the function $S_{\bs{L}}(\bs{Q})$ will be regarded as a function of the variables $Q_{1},\dotsc,Q_{\mc{N}}$, whereas the entries of the matrix $\bs{L}$ will be regarded as fixed parameters.} $Q_{1},\dotsc,Q_{\mc{N}}$ \cite[Corollary 1.2]{LeVaaler:Sumsof}. Their main result is the following.  

\begin{theorem}[L\^e-Vaaler]
For all $\bs{L}\in\mb{R}^{\mc{M}\times\mc{N}}$ and all $\bs{Q}\in[1,+\infty)^{\mc{N}}$ we have that
\begin{equation}
\label{eq:LeVaaler}
S_{\bs{L}}(\bs{Q})\gg_{\mc{M},\mc{N}}Q^{\mc{N}}(\log Q)^{\mc{M}},
\end{equation}
where $Q:=\left(Q_{1}\dotsm Q_{\mc{N}}\right)^{1/\mc{N}}$.
\end{theorem}

\noindent L\^e and Vaaler further asked whether Inequality (\ref{eq:LeVaaler}) is sharp, i.e., whether there exist matrices $\bs{L}\in\mb{R}^{\mc{M}\times\mc{N}}$ such that the reverse inequality holds \cite[Equations 1.4 and 1.5]{LeVaaler:Sumsof}. As a partial answer to this question, they proved that (\ref{eq:LeVaaler}) is sharp whenever the matrix $\bs{L}$ is multiplicatively badly approximable \cite[Theorem 2.1]{LeVaaler:Sumsof}. For the convenience of the reader, we recall the general definition of multiplicative bad approximability below.

\begin{definition}
\label{def:multbad}
Let $\phi:[0,+\infty)\to (0,1]$ be a non-increasing\footnote{We say that $\phi$ is non-increasing if $\phi(x)\geq\phi(y)$ for all $x<y$.} real function. We say that a matrix $\bs{L}\in\mb{R}^{\mc{M}\times\mc{N}}$ is $\phi$-multiplicatively badly approximable if
\begin{equation}
\label{eq:condition1}
\prod_{j=1}^{\mc{N}}\max\left\{1,|q_{j}|\right\}\prod_{i=1}^{\mc{M}}\|L_{i}\bs{q}\|\geq\phi\left(\left(\prod_{j=1}^{\mc{N}}\max\left\{1,|q_{j}|\right\}\right)^{\frac{1}{\mc{N}}}\right)
\end{equation}
for all $\bs{q}\in\mb{Z}^{\mc{N}}\setminus\{\bs{0}\}$. If $\phi$ is a constant function, we simply say that the matrix $\bs{L}$ is multiplicatively badly approximable.
\end{definition} 

\cite[Theorem 2.1]{LeVaaler:Sumsof} suggests that upper bounds for the sum $S_{\bs{L}}(\bs{Q})$ should depend on the multiplicative Diophantine type of the matrix $\bs{L}$, i.e., on the shape of the function $\phi$ in Definition \ref{def:multbad}. This is also corroborated by analogous results in lower-dimensional cases \cite{Fregoli:Sumsof}. However, the existence of multiplicatively badly approximable matrices for $\mc{M}+\mc{N}>2$ hasn't been established yet. In particular, it is well-known that each multiplicatively badly approximable matrix for $\mc{M}+\mc{N}>2$ would provide a counterexample to the famous Littlewood Conjecture \cite{Bugeaud:Around}, which to date remains unsolved. Due to the breakthrough achieved by Einsiedler, Lindenstrauss, and Katok \cite{EinseidlerKatokLindenstrauss:Invariant}, who showed that the set of counterexamples to this conjecture has zero Hausdorff dimension, a wide part of the mathematical community believes that the Littlewood Conjecture is true. Hence, \cite[Theorem 2.1]{LeVaaler:Sumsof} does not fully answer L\^e and Vaaler's  question on the sharpness of (\ref{eq:LeVaaler}). In particular, it would be desirable to prove the sharpness of their estimate unconditionally.

In \cite[Corollary 1.8]{Fregoli:Onacounting}, we improved on a special case of \cite[Theorem 2.1]{LeVaaler:Sumsof}. Namely, we showed that, whenever the condition $Q_{1}=\dotsb =Q_{\mc{N}}=Q$ holds, for any $\phi$-semimultiplicatively badly approximable\footnote{We recall that a matrix $\bs{L}\in\mb{R}^{\mc{M}\times\mc{N}}$ is $\phi$-semimultiplicatively badly approximable if
$|\bs{q}|_{\infty}^{\mc{N}}\prod_{i=1}^{\mc{M}}\|L_{i}\bs{q}\|\geq\phi\left(|\bs{q}|_{\infty}\right)$
for all $\bs{q}\in\mb{Z}^{\mc{N}}\setminus\{\bs{0}\}$. This condition is weaker than (\ref{eq:condition1}), since $\prod_{j=1}^{\mc{N}}\max\left\{1,|q_{j}|\right\}\leq|\bs{q}|_{\infty}$ for all $\bs{q}\in\mb{Z}^{\mc{N}}\setminus\{\bs{0}\}$.} matrix $\bs{L}$ we have the following upper bound:
\begin{equation}
\label{eq:Fre1}
S_{\bs{L}}(\bs{Q})\ll_{\mc{M},\mc{N}}Q^{\mc{N}}\log\left(\frac{Q}{\phi(Q)}\right)^{\mc{M}}+\frac{Q^{\mc{N}}}{\phi(Q)}\log\left(\frac{Q}{\phi(Q)}\right)^{\mc{M}-1}.
\end{equation}

This implies that, whenever the matrix $\bs{L}$ is $\phi$-semimultiplicatively badly approximable with $\phi(x)\gg\log(x)^{-1}$, the bound in (\ref{eq:LeVaaler}) becomes sharp for all values of $Q=Q_{1}=\dotsb Q_{\mc{N}}\geq 2$. Note that \cite[Theorem 2.1]{LeVaaler:Sumsof} requires $\phi$ to be a constant function, which is a much stronger hypothesis in comparison with $\phi(x)\gg\log(x)^{-1}$. More precisely, the existence of $\phi$-semimultiplicatively (and $\phi$-multiplicatively) badly approximable matrices $\bs{L}$ with $\phi(x)\gg\log(x)^{-1}$ is connected with a natural complement to the Littlewood Conjecture, discussed, e.g., in \cite{Fregoli:Multiplicatively}, \cite{BadziahinVelani:MultiplicativelyBadly}, and \cite{Badziahin:OnMultiplicatively}. In particular, continuum-many such matrices are conjectured to exist, showing the relevance of (\ref{eq:Fre1}) for the problem posed by L\^e and Vaaler.

In this paper, we prove that a similar estimate to (\ref{eq:Fre1}) holds for arbitrary values of the variables $Q_{1},\dotsc,Q_{\mc{N}}$ in $[1,+\infty)$, thus extending our previous result. However, while in the case when $Q_{1}=\dotsb =Q_{\mc{N}}$, semimultiplicative bad approximability is a robust enough condition to prove an upper bound for the sum $S_{\bs{L}}(\bs{Q})$, in the general case, we need to assume a stronger hypothesis. More specifically, we now assume that the matrix $\bs{L}$ is $\phi$-multiplicatively badly approximable. Note that this assumption does not affect our improvement on L\^e and Vaaler upper bounds, since the considerations in \cite{BadziahinVelani:MultiplicativelyBadly} and \cite{Fregoli:Multiplicatively} hold both for $\phi$-semimultiplicatively and for $\phi$-multiplicatively badly approximable matrices with $\phi(x)\gg\log(x)^{-1}$.

The following result represents the main contribution of this paper.

\begin{theorem}
\label{thm:cor2}
Let $\bs{L}\in\mb{R}^{\mc{M}\times\mc{N}}$ be a $\phi$-multiplicatively badly approximable matrix. Then, we have that
\begin{equation}
S_{\bs{L}}(\bs{Q})\ll_{\mc{M},\mc{N}}Q^{\mc{N}}\log\left(\frac{Q}{\phi(Q)}\right)^{\mc{M}}+\frac{Q^{\mc{N}}}{\phi(Q)}\log\left(\frac{Q}{\phi(Q)}\right)^{\mc{M}-1}\nonumber
\end{equation}
for all $\bs{Q}\in[1,+\infty)^{\mc{N}}$, provided $Q:=\left(Q_{1}\dotsm Q_{\mc{N}}\right)^{1/\mc{N}}\geq 2$.
\end{theorem}
 
Note how this estimate only depends on the geometric mean of the variables $Q_{1},\dotsc,Q_{\mc{N}}$ and not on their maximum $Q_{\max}:=\max\{Q_{1},\dotsc,Q_{\mc{N}}\}$. In other words, only the volume of the box where the variable $\bs{q}$ ranges in (\ref{eq:intro1}) determines upper bounds for the function $S_{\bs{L}}(\bs{Q})$. This property, which also characterises L\^e and Vaaler's lower bounds (\ref{eq:LeVaaler}), requires a substantially more elaborate proof in comparison with the case when $Q_{1}=\dotsb =Q_{\mc{N}}$ \cite[Corollary 1.8]{Fregoli:Onacounting}. In fact, the techniques used in \cite{Fregoli:Onacounting} are insufficient to obtain a bound independent of the function $Q_{\max}$. In the remainder of this section, we attempt to give a quick insight into the distinctive features of the non-symmetric case (i.e., when the variables $Q_{i}$ take arbitrary values), in order to highlight the novelty of our approach to this problem.

\subsection{Broad Brush Description of Methods}

To prove Theorem \ref{thm:cor2}, we estimate the cardinality of the following typical sets:
\begin{multline*}
M(\bs{L},\varepsilon,T,\bs{Q}):=\left\{(\bs{p},\bs{q})\in\mb{Z}^{\mc{M}}\times\left(\mb{Z}^{\mc{N}}\setminus\{\bs{0}\}\right):\prod_{i=1}^{\mc{M}}\left|L_{i}\bs{q}+p_{i}\right|<\varepsilon,\right.\\
|L_{i}\bs{q}+p_{i}|\leq T,\ i=1,\dotsc,\mc{M},\ \ |q_{j}|\leq Q_{j},\ j=1,\dotsc,\mc{N}\Bigg\},
\end{multline*}
where $\varepsilon,T\in(0,+\infty)$. Once an estimate of the sort is established, the proof of Theorem \ref{thm:cor2} will follow directly from a standard dyadic argument, presented in Section \ref{sec:proofofthm}.

Estimating the quantity $\#M(\bs{L},\varepsilon,T,\bs{Q})$ can be reduced to a lattice-point counting problem. We show how this is done in the next few lines. Henceforth, we write $\bs{v}=(\bs{x},\bs{y})$, with $\bs{x}\in\mb{R}^{\mc{M}}$ and $\bs{y}\in\mb{R}^{\mc{N}}$, for vectors $\bs{v}\in\mb{R}^{\mc{M}+\mc{N}}$. We let
$$ \bs{A}_{\bs{L}}:=\left(\begin{array}{@{}c|c@{}}
    \bs{I}_{\mc{M}} & \!\!\!\!\bs{L} \\\hline
    \bs{0} & \accentset{\phantom{a}}{\bs{I}}_{\mc{N}}
  \end{array}\right)\in\mb{R}^{(\mc{M}+\mc{N})\times(\mc{M}+\mc{N})},$$
where $\bs{I}_{\mc{M}}$ and $\bs{I}_{\mc{N}}$ are identity matrices of size $\mc{M}$ and $\mc{N}$ respectively, and we let $\Lambda_{\bs{L}}:=\bs{A}_{\bs{L}}\mb{Z}^{\mc{M}+\mc{N}}\subset\mb{R}^{\mc{M}+\mc{N}}$. We further define the sets
\begin{equation*}
H:=\left\{\bs{x}\in\mb{R}^{\mc{M}}:\prod_{i= 1}^{\mc{M}}\left|x_{i}\right|<\varepsilon,\ |x_{i}|\leq T,\ i=1,\dotsc,\mc{M}\right\}
\end{equation*}
and $Z:=H\times\prod_{j=1}^{\mc{N}}[-Q_{j},Q_{j}]\subset\mb{R}^{\mc{M}+\mc{N}}$. Then, we have that
\begin{equation}
\label{eq:intersection}
\# M(\bs{L},\varepsilon,T,\bs{Q})=\#\left(\left(\Lambda_{\bs{L}}\cap Z\right)\setminus C\right),
\end{equation}
where $C:=\{\bs{y}=\bs{0}\}\subset\mb{R}^{\mc{M}+\mc{N}}$. Since $\Lambda_{\bs{L}}\cap C=\mb{Z}^{\mc{M}}\times\{\bs{0}\}$, to estimate $\# M(\bs{L},\varepsilon,T,\bs{Q})$, it simply suffices to determine the cardinality of the set $\Lambda_{\bs{L}}\cap Z$. Thus, the problem is reduced to counting the points of a particular unipotent lattice, induced by the matrix $\bs{L}$, inside a family of non-convex sets $Z$, defined by the parameters $\varepsilon$ and $T$.

Up to Equation (\ref{eq:intersection}), the strategy presented here coincides with that presented in \cite{Fregoli:Onacounting}, which the reader may keep in mind as a useful term of comparison. However, in this setting, we can no longer apply \cite[Theorem 1.5]{Fregoli:Onacounting} to estimate $\#(\Lambda_{\bs{L}}\cap Z)$. This follows from the fact that the lattice $\Lambda_{\bs{L}}$ does not have the right type of weak admissibility to apply this particular result (see \cite[Definition 1.3]{Fregoli:Onacounting}). More precisely, in order to obtain an estimate that would feature only the geometric mean of the variables $Q_{1},\dotsc,Q_{\mc{N}}$ via \cite[Theorem 1.5]{Fregoli:Onacounting}, the lattice $\Lambda_{\bs{L}}$ would have to be weakly admissible for the couple $((\bs{l},\bs{\alpha}),C)$, where $\bs{l}=\bs{\alpha}:=(1,\dotsc,1)\in\mb{R}^{\mc{M}+\mc{N}}$ and $C:=\{\bs{y}=\bs{0}\}$. This would imply that for all $(\bs{p},\bs{q})\in\mb{Z}^{\mc{M}}\times\left(\mb{Z}^{\mc{N}}\setminus\{\bs{0}\}\right)$ we have that
$$\prod_{j=1}^{\mc{N}}|q_{j}|\prod_{i=1}^{\mc{M}}|L_{i}\bs{q}+p_{i}|>0,$$
which does not necessarily hold, due to the maximum functions appearing in the definition of multiplicative bad approximability (see (\ref{eq:condition1})). Nonetheless, we can obtain an upper bound for the quantity $\#(\Lambda_{\bs{L}}\cap Z)$, even though not an asymptotic estimate as in the symmetric case (compare with \cite[Proposition 1.2]{Fregoli:Onacounting}), by using a more subtle argument.
 
\begin{prop}
\label{prop:cor1}
Let $\bs{L}\in\mb{R}^{\mc{M}\times \mc{N}}$ be a $\phi$-multiplicatively badly approximable matrix and suppose that $T^{\mc{M}}/\varepsilon\geq e^{\mc{M}}$, where $e=2.71828\dots$ is the base of the natural logarithm. Then,
\begin{equation}
\label{eq:cor1}
\#\left(\left(\Lambda_{\bs{L}}\cap Z\right)\setminus C\right)
\ll_{\mc{M},\mc{N}}(1+T)^{\mc{M}+\mc{N}-1}\log\left(\frac{T^{\mc{M}}}{\varepsilon}\right)^{\mc{M}-1}\left[\varepsilon Q^{\mc{N}}+\left(\frac{\varepsilon Q^{\mc{N}}}{\phi\left(Q\right)}\right)^{\frac{\mc{M}+\mc{N}-1}{\mc{M}+\mc{N}}}\right],
\end{equation}
where $Q:=\left(Q_{1}\dotsm Q_{\mc{N}}\right)^{1/\mc{N}}$.
\end{prop}
\noindent Theorem \ref{thm:cor2} is a straightforward consequence of Proposition \ref{prop:cor1}, as shown in Section \ref{sec:proofofthm}.

Now, as mentioned in the first lines of this paper, Proposition \ref{prop:cor1} follows form the combination of two key ingredients. One is \cite[Proposition 2.1]{Fregoli:Onacounting}, which is the essence of a new partition method developed by Widmer and the author to count lattice points in sets with "hyperbolic spikes", such as $Z$. In this method, we partition the set $Z$ into approximately $\log(T/\varepsilon)^{\mc{M}}$ sets, and we apply different unimodular diagonal linear maps to each set of the partition, to transform it into a ball centred at the origin and having fixed radius. Then, we count the points of the different mappings of the lattice $\Lambda_{\bs{L}}$ lying within this ball. To do so, we rely on the second key ingredient to our proof, i.e., Proposition \ref{prop:firstminima}, where the successive minima of the different mappings of the lattice $\Lambda_{\bs{L}}$ are estimated. This proposition replaces \cite[Proposition 2.5]{Fregoli:Onacounting}, where we considered only the first successive minimum of the transformed lattices. To overcome the fact that the lattice $\Lambda_{\bs{L}}$ is no longer weakly admissible, we need to use a more sophisticated approach, taking into account the higher successive minima of the different mappings of $\Lambda_{\bs{L}}$. Prior to giving details, we roughly explain how this approach works.

Assume that we have partitioned the set $Z$ into smaller sets $Z_{k}$ and let $\varphi_{k}$ be a unique unimodular diagonal linear map associated with the region $Z_{k}$ that transforms it into a ball of given radius. In order to obtain the estimate in (\ref{eq:cor1}), we bound above the quantity $\#(\varphi_{k}(\Lambda_{\bs{L}})\cap\varphi_{k}(Z_{k}))$ for each $k$ and then we sum these quantities for $Z_{k}$ ranging in the given partition of the set $Z$. By standard considerations in the Geometry of Numbers, the quantity $\#(\varphi_{k}(\Lambda_{\bs{L}})\cap\varphi_{k}(Z_{k}))$ can be bounded above by summing the ratios
\begin{equation}
\label{eq:prodexample}
\frac{\mathscr{V}_{s}(\varepsilon,Q)}{\lambda_{1}\dotsm\lambda_{s}}
\end{equation} 
for $s=1,\dotsc,\mc{M}+\mc{N}$, where $\mathscr{V}_{s}(\varepsilon,Q)$ is an upper bound for the volume of the region $\varphi_{k}(Z_{k})$, and $\lambda_{1},\dotsc,\lambda_{s}$ are the first $s$ successive minima of the lattice $\varphi_{k}(\Lambda_{\bs{L}})$. In \cite{Fregoli:Onacounting}, we rely on the inequality $\lambda_{1}\dotsm\lambda_{s}\geq\lambda_{1}^{s}$, and we estimate only the first successive minimum $\lambda_{1}$ of $\varphi_{k}(\Lambda_{\bs{L}})$, thus finding an upper bound for the quantity $\#(\varphi_{k}(\Lambda_{\bs{L}})\cap\varphi_{k}(Z_{k}))$. In the case of distorted boxes, this approach typically fails. Indeed, suppose that $\bs{v}^{*}=(\bs{x}^{*},\bs{y}^{*})$ is a vector in the lattice $\Lambda_{\bs{L}}$ such that $|\varphi_{k}\left(\bs{v}^{*}\right)|_{2}=\lambda_{1}$. Then, $\bs{v}^{*}\in\mb{R}^{\mc{M}+\mc{N}}$ has the form
$$\bs{v}^{*}:=(\underbrace{L_{1}\bs{q}^{*}+p^{*}_{1},\dotsc,L_{\mc{M}}\bs{q}^{*}+p^{*}_{\mc{M}}}_{=\bs{x}^{*}},\underbrace{\bs{q}^{*}}_{=\bs{y}^{*}})$$
for some $\bs{p}^{*}\in\mb{Z}^{\mc{M}}$ and $\bs{q}^{*}\in\mb{Z}^{\mc{N}}\setminus\{\bs{0}\}$ (we can assume $\bs{q}^{*}\neq\bs{0}$, since otherwise the proof is trivial). Assume that for all $\bs{v}\in\mb{R}^{\mc{M}+\mc{N}}$
$$\varphi_{k}(\bs{v}):=(\mu_{1}x_{1},\dotsc,\mu_{\mc{M}}x_{\mc{M}},\nu_{1}y_{1},\dotsc,\nu_{\mc{N}}y_{\mc{N}}),$$
where the constants $\mu_{i},\nu_{j}>0$ are such that $\mu_{1}\dotsc\mu_{\mc{M}}\nu_{1}\dotsb\nu_{\mc{N}}=1$. Then, we have that

\begin{equation}
\lambda_{1}=\left|\varphi_{k}\left(\bs{v}^{*}\right)\right|_{2}=\left(\sum_{i=1}^{\mc{M}}|\mu_{i}x^{*}_{i}|^{2}+\sum_{j=1}^{\mc{N}}|\nu_{j}y^{*}_{j}|^{2}\right)^{1/2}.\nonumber
\end{equation}

Since we know that the product $\mu_{1}\dotsc\mu_{\mc{M}}\nu_{1}\dotsb\nu_{\mc{N}}$ of the rescaling factors in the map $\varphi_{k}$ is $1$, we can use an arithmetic-geometric mean inequality to bound below the value of $\lambda_{1}$, as we do in \cite{Fregoli:Onacounting}. This inequality, however, heavily depends on the second Cartesian factor of the set $Z$ and, more precisely, on the relative size of the parameters $Q_{1},\dotsc,Q_{\mc{N}}$. Let us be more precise. As in \cite{Fregoli:Onacounting}, the sets $Z_{k}$ are Cartesian products. In particular, they take the shape  
$$Z_{k}:=H_{k}\times\prod_{j=1}^{\mc{N}}[-Q_{j},Q_{j}],$$
where $H_{k}$ derives from a partition of the set $H$. 
If $Q_{1}=\dotsb =Q_{\mc{N}}$, then, to transform the set $Z_{k}$ into a ball, it suffices to define $\nu_{1}=\dotsb =\nu_{\mc{N}}=:\nu$. We can then write
\begin{multline}
\left|\varphi_{k}\left(\bs{v}^{*}\right)\right|_{2}=\left(\sum_{i=1}^{\mc{M}}|\mu_{i}x^{*}_{i}|^{2}+\sum_{j=1}^{\mc{N}}|\nu y^{*}_{j}|^{2}\right)^{1/2}\geq\left(\prod_{i=1}^{\mc{M}}|\mu_{i}x^{*}_{i}|\cdot\nu^{\mc{N}}|\bs{y}^{*}|_{2}^{\mc{N}}\right)^{\frac{1}{\mc{M}+\mc{N}}} \\
=\left(\prod_{i=1}^{\mc{M}}|x^{*}_{i}|\cdot|\bs{y}^{*}|_{2}^{\mc{N}}\right)^{\frac{1}{\mc{M}+\mc{N}}}=\left(\prod_{i=1}^{\mc{M}}|L_{i}\bs{q}^{*}+p^{*}_{i}|\cdot|\bs{q}^{*}|_{2}^{\mc{N}}\right)^{\frac{1}{\mc{M}+\mc{N}}}\geq \left(\prod_{i=1}^{\mc{M}}\|L_{i}\bs{q}^{*}\|\cdot|\bs{q}^{*}|_{2}^{\mc{N}}\right)^{\frac{1}{\mc{M}+\mc{N}}}.\nonumber
\end{multline}
Crucially, the last term of this inequality does not depend on the specific set $Z_{k}$ in the partition and can be bounded below by values of the function $\phi$, thus deducing the required bound. On the other hand, when the variables $Q_{1},\dotsc,Q_{\mc{N}}$ take arbitrary values, we can no longer assume that $\nu_{1}=\dotsb =\nu_{\mc{N}}$. Let $1\leq h^{*}\leq\mc{N}$ be the largest non-null index such that $q^{*}_{h^{*}}\neq 0$. Then, we can write

\begin{equation}
\label{eq:productexplanation}
\left|\varphi_{k}\left(\bs{v}^{*}\right)\right|_{2}=\left(\sum_{i=1}^{\mc{M}}|\mu_{i}x^{*}_{i}|^{2}+\sum_{j=1}^{\mc{N}}|\nu_{j}y^{*}_{j}|^{2}\right)^{1/2}
\geq\left(\prod_{i=1}^{\mc{M}}|\mu_{i}x^{*}_{i}|\prod_{j=1}^{h^{*}}|\nu_{j}y^{*}_{j}|\right)^{1/(\mc{M}+h^{*})},
\end{equation}

but we have no control over the number $h^{*}$. In particular, if $h^{*}<\mc{N}$, we cannot fully exploit the fact that the product of the rescaling factors $\mu_{1}\dotsm\mu_{\mc{M}}\nu_{1}\dotsm\nu_{\mc{N}}$ in the map $\varphi_{k}$ is equal to $1$. This yields an extra factor of $(\nu_{h^{*}+1}\dotsc\nu_{\mc{N}})^{-1}$ on the right-hand side of (\ref{eq:productexplanation}), which increases the exponents of the variables $Q_{h^{*}+1},\dotsc,Q_{\mc{N}}$ in the overall lower bound for the minimum $\lambda_{1}$. Hence, our estimate becomes dependent on the value of $Q_{\max}$, i.e., the largest component of the vector $\bs{Q}$, rather than the geometric mean of all its components. This dependence is obviously insignificant in the case when $Q_{1}=\dotsb =Q_{\mc{N}}$, but it becomes non-negligible for highly distorted boxes, given that the function $Q_{\max}:=\max\{Q_{1},\dotsc,Q_{\mc{N}}\}$ can be much larger than the function $Q:=(Q_{1}\dotsm Q_{\mc{N}})^{1/\mc{N}}$.

To solve the issue presented above, we consider all the successive minima of the lattice $\varphi_{k}(\Lambda_{\bs{L}})$ simultaneously. By a theorem of Mahler and Weyl, there is a basis $\{\bs{v}^{1}\dotsc,\bs{v}^{\mc{M}+\mc{N}}\}$ of the lattice $\varphi_{k}(\Lambda_{\bs{L}})$ such that the vector $\bs{v}^{s}=(L_{1}\bs{q}^{s}+p^{s}_{1},\dotsc,L_{\mc{M}}\bs{q}+p^{s}_{\mc{M}},\bs{q}^{s})$ is (up to a constant) of length $\lambda_{s}$ for $s=1,\dotsc,\mc{M}+\mc{N}$. Since the vectors $\bs{v}^{1},\dotsc,\bs{v}^{\mc{M}+\mc{N}}$ are linearly independent, there cannot be an index $j$ with $1\leq j\leq\mc{N}$ such that $q^{s}_{j}=0$ for all $s$. Hence, we can balance out the presence of the constant $\nu_{j}$ in the products $\lambda_{1}\dotsm\lambda_{s}\sim|\bs{v}^{1}|_{2}\dotsm|\bs{v}^{s}|_{2}$ for $s=1,\dotsc,\mc{M}+\mc{N}$, and obtain a sufficiently good estimate of the terms in (\ref{eq:prodexample}). By doing so, we recover the factor $Q=(Q_{1}\dotsm Q_{\mc{N}})^{1/\mc{N}}$ in our upper bound. This constitutes the main novelty in our approach.

\section{Proof of Proposition \ref{prop:cor1}}

This proof has a similar structure to the proof Theorem 1.5 in \cite{Fregoli:Onacounting}. Nonetheless, we recall most details for the sake of completeness. The crucial difference between the two proofs is represented by Proposition \ref{prop:firstminima}.

We start off by partitioning the set $Z$. Let
$$H_{+}:=H\cap\left\{\bs{x}\in\mb{R}^{\mc{M}}:x_{i}\neq 0,\ i=1,\dotsc,\mc{M}\right\}$$
and let $Z_{+}:=H_{+}\times\prod_{j=1}^{\mc{N}}[-Q_{j},Q_{j}]$. Let also
$$H^{i}:=H\cap\{x_{i}=\bs{0}\}$$
and $Z^{i}:=H^{i}\times\prod_{j=1}^{\mc{N}}[-Q_{j},Q_{j}]$ for $i=1,\dotsc,\mc{M}$. Then, we have
\begin{equation}
Z=Z_{+}\cup\bigcup_{i=1}^{\mc{M}}Z^{i}.\nonumber
\end{equation}
It follows that
\begin{equation}
\#(\Lambda_{\bs{L}}\cap Z)\leq\#(\Lambda_{\bs{L}}\cap Z_{+})+\sum_{i=1}^{\mc{M}}\#\left(\Lambda_{\bs{L}}\cap Z^{i}\right).\nonumber
\end{equation}

Now, we apply \cite[Proposition 2.1]{Fregoli:Onacounting} with $\bs{m}=\bs{\beta}:=\overbrace{(1,\dotsc,1)}^{\mc{M}\ \mbox{times}}$ and parameters $T$ and $\varepsilon^{1/\mc{M}}$, to conveniently partition the set $H_{+}$. For the sake of precision, we restate this proposition below. Note that in Proposition \ref{thm:partition} the symbol $a_{i}^{k}$ stands for a sequence indexed by $i$ and $k$ and not for the number $a_{i}$ raised to the $k$-th power.

\begin{prop}
\label{thm:partition}
Suppose that $T^{\mc{M}}/\varepsilon>e^{\mc{M}}$, where $e=2.71828\dots$ is the base of the natural logarithm. Then, there exists a partition $H_{+}=\bigcup_{k\in\mc{K}}\! X_{k}$ of the set $H_{+}$ and there exists a collection of linear maps $\left\{\varphi_{k}\right\}_{k\in{\mc{K}}}$ defined over $\mb{R}^{\mc{M}}$
such that\vspace{2mm}
\begin{itemize}
\item[$i)$] $\#\mc{K}\ll_{\mc{M}}\log\left(T/\varepsilon^{1/\mc{M}}\right)^{\mc{M}-1}$;\vspace{2mm}
\item[$ii)$] the maps $\varphi_{k}$ for $k\in\mc{K}$ are determined by the expression $\varphi_{k}(\bs{x})_{i}:=\exp\left(a^{k}_{i}-c\right)x_{i}$ for
$i=1,\dotsc,\mc{M}$, where $c\in\mb{R}$ is a constant only depending on $\mc{M}$, and the coefficients $a^{k}_{i}\in\mb{R}$ satisfy\vspace{2mm}
\begin{itemize}
\item[$iia)$] $\exp\left(a^{k}_{i}-c\right)\gg_{\mc{M}}\varepsilon^{1/\mc{M}}/T$ for $i=1,\dotsc,\mc{M}$;\vspace{2mm}
\item[$iib)$] $\sum_{i=1}^{\mc{M}}a^{k}_{i}=0$;
\end{itemize}
\vspace{2mm}
\item[$iii)$] $\varphi_{k}\left(X_{k}\right)\subset\left[-\varepsilon^{1/\mc{M}},\varepsilon^{1/\mc{M}}\right]^{\mc{M}}$ for all $k\in\mc{K}$.\vspace{2mm}
\end{itemize}
\end{prop}

We set $\hat{X}_{k}:=X_{k}\times\prod_{j=1}^{\mc{N}}[-Q_{j},Q_{j}]$ for $k\in\mc{K}$, and we extend the maps $\varphi_{k}$ to $\hat{\varphi}_{k}:\mb{R}^{\mc{M}+\mc{N}}\to\mb{R}^{\mc{M}+\mc{N}}$ ($k\in\mc{K}$), by defining $\hat{\varphi}_{k}$ as the identity map on the second $\mc{N}$ coordinates. In view of this, we find a partition
$$Z_{+}=\bigcup_{k\in\mc{K}}\hat{X}_{k}$$
of the set $Z_{+}$. Hence, we have
\begin{align}
 \#\left(\Lambda_{\bs{L}}\cap Z\right) & \leq\#(\Lambda_{\bs{L}}\cap Z_{+})+\sum_{i=1}^{\mc{M}}\#\left(\Lambda_{\bs{L}}\cap Z^{i}\right)\nonumber\\
 & =\sum_{k\in \mc{K}}\#\left(\Lambda_{\bs{L}}\cap \hat{X}_{k}\right)+\sum_{i=1}^{\mc{M}}\#\left(\Lambda_{\bs{L}}\cap Z^{i}\right)\nonumber \\
 & =\sum_{k\in\mc{K}}\#\left(\hat{\varphi}_{k}(\Lambda_{\bs{L}})\cap \hat{\varphi}_{k}\left(\hat{X}_{k}\right)\right)+\sum_{i=1}^{\mc{M}}\#\left(\Lambda_{\bs{L}}\cap Z^{i}\right)\nonumber.
\end{align}
\noindent We deal with these terms separately. We start with $\#\left(\Lambda_{\bs{L}}\cap Z^{i}\right)$ for $i=1,\dotsc,\mc{M}$.

\begin{lem}
\label{lem:LambdacapC}
For $i=1,\dotsc,\mc{M}$ we have
\begin{equation}
\#\left(\Lambda_{\bs{L}}\cap Z^{i}\right)\ll_{\mc{M}}(1+T)^{\mc{M}}.\nonumber
\end{equation}
\end{lem}

\begin{proof}
Since the entries $L_{i,1},\dotsc,L_{i\mc{N}}$ along with $1$ are linearly independent over $\mb{Z}$, the equation $L_{i}\bs{q}+p_{i}=0$ implies that $\bs{q}=\bs{0}$. It follows that $\Lambda_{\bs{L}}\cap Z^{i}\subset \Lambda_{\bs{L}}\cap C$ for $i=1,\dotsc,\mc{M}$. Hence, we have that $\#(\Lambda_{\bs{L}}\cap Z^{i})\leq\#((\Lambda_{\bs{L}}\cap C)\cap(Z\cap C))$. Now, we observe that $\Lambda_{\bs{L}}\cap C=\mb{Z}^{\mc{M}}\times\{\bs{0}\}$ and $Z\cap C\subset[-T,T]^{\mc{M}}\times\{\bs{0}\}$. This immediately yields the required inequality.
\end{proof}

We are left to estimate the quantity $\#\left(\hat{\varphi}_{k}(\Lambda_{\bs{L}})\cap \hat{\varphi}_{k}\left(\hat{X}_{k}\right)\right)$ for $k\in\mc{K}$. To do so, we use a powerful result from \cite{Widmer:CountingLattice} which generalises a theorem of Davenport (see \cite[Equation (1.2) and discussion therein]{Widmer:CountingLattice}).
\begin{theorem}
\label{thm:latticeest}
Let $n\in\mb{N}$ and let $\Lambda$ be a full rank lattice in $\mb{R}^{n}$. Let also $\bs{P}\in[0,+\infty)^{n}$ and $B_{\bs{P}}:=\prod_{i=1}^{n}[-P_{i},P_{i}]\subset\mb{R}^{n}$. Then, we have that
$$\#(\Lambda\cap B_{\bs{P}})\ll_{n}1+\sum_{s=1}^{n}\frac{\mathscr{V}_{s}(B_{\bs{P}})}{\lambda_{1}\dotsm\lambda_{s}},$$
where $\mathscr{V}_{s}(B_{\bs{P}}):=\max\{P_{i_{1}}\dotsm P_{i_{s}}:1\leq i_{1}<\dotsb < i_{s}\leq n\}$, and $\lambda_{s}$ is the $s$-th successive minimum of the lattice $\Lambda$.
\end{theorem}

From Proposition \ref{thm:partition} it follows that
\begin{equation}
\label{eq:27}
\hat{\varphi}_{k}\left(\hat{X}_{k}\right)\subset\left[-\varepsilon^{1/\mc{M}},\varepsilon^{1/\mc{M}}\right]^{\mc{M}}\times\prod_{j=1}^{\mc{N}}[-Q_{j},Q_{j}]
\end{equation}
for $k\in\mc{K}$. Before applying Theorem \ref{thm:latticeest}, we rescale the box in (\ref{eq:27}) and transform it into a cube, to gain a better control over the value of $\mathscr{V}_{s}$. We set
$$\theta:=\frac{\left(\varepsilon Q^{\mc{N}}\right)^{\frac{1}{\mc{M}+\mc{N}}}}{\varepsilon^{1/\mc{M}}},$$
and we consider two linear maps $\omega_{1},\omega_{2}:\mb{R}^{\mc{M}+\mc{N}}\to \mb{R}^{\mc{M}+\mc{N}}$, defined by
$$\omega_{1}(\bs{x},\bs{y}):=\left(\bs{x},\frac{Q}{Q_{1}}y_{1},\dotsc,\frac{Q}{Q_{\mc{N}}}y_{\mc{N}}\right)$$
and
$$\omega_{2}\left(\bs{x},\bs{y}\right):=\left(\theta\bs{x},\theta^{-\frac{\mc{M}}{\mc{N}}}\bs{y}\right)$$
for $(\bs{x},\bs{y})\in\mb{R}^{\mc{M}+\mc{N}}$.
By applying these maps to the box in question, we obtain that
\begin{equation}
\omega_{2}\circ\omega_{1}\left(\left[-\varepsilon^{1/\mc{M}},\varepsilon^{1/\mc{M}}\right]^{\mc{M}}\times\prod_{j=1}^{\mc{N}}[-Q_{j},Q_{j}]\right)=\left[-\left(\varepsilon Q^{\mc{N}}\right)^{\frac{1}{\mc{M}+\mc{N}}},\left(\varepsilon Q^{\mc{N}}\right)^{\frac{1}{\mc{M}+\mc{N}}}\right]^{\mc{M}+\mc{N}}.\nonumber
\end{equation}
\noindent Hence, from Lemma \ref{lem:LambdacapC} and Theorem \ref{thm:latticeest} we deduce that
\begin{align}
\label{neweq:estimate3}
 & \#\left(\Lambda_{\bs{L}}\cap Z\right)\leq\sum_{k\in\mc{K}}\#\left(\omega_{2}\circ\omega_{1}\circ\hat{\varphi}_{k}(\Lambda_{\bs{L}})\cap \omega_{2}\circ\omega_{1}\circ\hat{\varphi}_{k}\left(\hat{X}_{k}\right)\right)+\sum_{i=1}^{\mc{M}}\#\left(\Lambda_{\bs{L}}\cap Z^{i}\right)\nonumber\\
 & \ll_{\mc{M},\mc{N}}\sum_{k\in\mc{K}}\left(1+\sum_{s=1}^{\mc{M}+\mc{N}}\frac{\left(\varepsilon Q^{\mc{N}}\right)^{\frac{s}{\mc{M}+\mc{N}}}}{\lambda_{1}\left(\omega_{2}\circ\omega_{1}\circ\hat{\varphi}_{k}(\Lambda_{\bs{L}})\right)\dotsm\lambda_{s}\left(\omega_{2}\circ\omega_{1}\circ\hat{\varphi}_{k}(\Lambda_{\bs{L}})\right)}\right)+(1+T)^{\mc{M}},
\end{align}
where $\lambda_{s}\left(\omega_{2}\circ\omega_{1}\circ\hat{\varphi}_{k}(\Lambda_{\bs{L}})\right)$ is the $s$-th successive minimum of the lattice $\omega_{2}\circ\omega_{1}\circ\hat{\varphi}_{k}(\Lambda_{\bs{L}})$ ($k\in\mc{K}$) for $s=1,\dotsc,\mc{M}+\mc{N}$.
We are therefore left to estimate the quantities
$$\frac{\left(\varepsilon Q^{\mc{N}}\right)^{\frac{s}{\mc{M}+\mc{N}}}}{\lambda_{1}\left(\omega_{2}\circ\omega_{1}\circ\hat{\varphi}_{k}(\Lambda_{\bs{L}})\right)\dotsm\lambda_{s}\left(\omega_{2}\circ\omega_{1}\circ\hat{\varphi}_{k}(\Lambda_{\bs{L}})\right)}$$
for $k\in\mc{K}$. We do this in the following proposition, which represents the crucial novelty of this paper in comparison with \cite{Fregoli:Onacounting}.

\begin{prop}
\label{prop:firstminima}
Let $k\in\mc{K}$ and let $\lambda_{1},\dotsc,\lambda_{\mc{M}+\mc{N}}$ be the successive minima of the lattice
$\omega_{2}\circ\omega_{1}\circ\hat{\varphi}_{k}(\Lambda_{\bs{L}})$. Then,
\begin{align}\
\label{eq:firstminima} 
 & \frac{\left(\varepsilon Q^{\mc{N}}\right)^{\frac{s}{\mc{M}+\mc{N}}}}{\lambda_{1}\dotsm\lambda_{s}}\ll_{\mc{M},\mc{N}}1+T^{\mc{M}+\mc{N}-1}+\varepsilon Q^{\mc{N}}+\left(\frac{\varepsilon Q^{\mc{N}}}{\phi(Q)}\right)^{\frac{\mc{M}+\mc{N}-1}{\mc{M}+\mc{N}}}.
\end{align}
for all $s=1,\dotsc,\mc{M}+\mc{N}$.
\end{prop}

\noindent Combining (\ref{neweq:estimate3}) and Proposition \ref{prop:firstminima}, we find that
\begin{equation}
\#\left(\Lambda_{\bs{L}}\cap Z\right)\ll_{\mc{M},\mc{N}}\#\mc{K}\left((1+T)^{\mc{M}+\mc{N}-1}+\varepsilon Q^{\mc{N}}+\left(\frac{\varepsilon Q^{\mc{N}}}{\phi(Q)}\right)^{\frac{\mc{M}+\mc{N}-1}{\mc{M}+\mc{N}}}\right).\nonumber
\end{equation}
This, in combination with Equation (\ref{eq:intersection}) and the fact that $\#\left(\Lambda_{\bs{L}}\cap C\right)\leq(1+T)^{\mc{M}}$ (see proof of Lemma \ref{lem:LambdacapC}) implies that
\begin{equation}
\label{eq:almostlasteq2}
\#M(\bs{L},\varepsilon,T,\bs{Q})\ll_{\mc{M},\mc{N}}\#\mc{K}\left((1+T)^{\mc{M}+\mc{N}-1}+\varepsilon Q^{\mc{N}}+\left(\frac{\varepsilon Q^{\mc{N}}}{\phi(Q)}\right)^{\frac{\mc{M}+\mc{N}-1}{\mc{M}+\mc{N}}}\right).
\end{equation}
In view of Proposition \ref{thm:partition}, we also have that $\#\mc{K}\ll_{\mc{M}}\log\left(T/\varepsilon^{1/\mc{M}}\right)^{\mc{M}-1}$. Therefore, if $\varepsilon Q^{\mc{N}}/\phi(Q) \\ \geq 1$, the required estimate is a straightforward consequence of (\ref{eq:almostlasteq2}). We are then left to prove the claim in the simpler case when $\varepsilon Q^{\mc{N}}/\phi(Q)<1$. To do this, we make use of the following lemma.

\begin{lem}
\label{lem:emptycase}
Assume that $\varepsilon Q^{\mc{N}}/\phi(Q)<1$. Then, we have that $\Lambda_{\bs{L}}\cap Z\subset C$.
\end{lem}

\begin{proof}
Suppose by contradiction that there exists a vector $\bs{v}\in(\Lambda_{\bs{L}}\cap Z)\setminus C$. Then, we can write
$$\bs{v}=(L_{1}\bs{q}+p_{1},\dotsc,L_{\mc{M}}\bs{q}+p_{\mc{M}},\bs{q})$$
for some $\bs{p}\in\mb{Z}^{\mc{M}}$ and $\bs{q}\in\mb{Z}^{\mc{N}}\setminus\{\bs{0}\}$. However, since $\bs{v}\in Z$, we have that
\begin{multline}
\prod_{j=1}^{\mc{N}}\max\left\{1,|q_{j}|\right\}\prod_{i=1}^{\mc{M}}\left\|L_{i}\bs{q}\right\|\leq \prod_{j=1}^{\mc{N}}\max\left\{1,|q_{j}|\right\}\prod_{i=1}^{\mc{M}}\left|L_{i}\bs{q}+p_{i}\right| \\
\leq Q^{\mc{N}}\varepsilon<\phi(Q)\leq\phi\left(\left(\prod_{j=1}^{\mc{N}}\max\left\{1,|q_{j}|\right\}\right)^{\frac{1}{\mc{N}}}\right),\nonumber
\end{multline}
in contradiction with (\ref{eq:condition1}). This proves the claim.
\end{proof}
By Lemma \ref{lem:emptycase} and (\ref{eq:intersection}), if $\varepsilon Q^{\mc{N}}/\phi(Q)< 1$, we have that $M(\bs{L},\varepsilon,T,\bs{Q})=\emptyset$, and (\ref{eq:cor1}) becomes trivial. Hence, the proof is complete.

\section{Proof of Proposition \ref{prop:firstminima}}

\subsection{Preliminaries}

In this section and in the rest of the paper, we use the phrase "basis of a lattice", e.g., $\Lambda\subset\mb{R}^{n}$, to denote a set of linearly independent vectors in $\mb{R}^{n}$ which generates the lattice $\Lambda$ as an Abelian subgroup of $\mb{R}^{n}$.

The goal of this subsection is to construct a "workable" basis for the lattice $\omega_{2}\circ\omega_{1}\circ\hat{\varphi}_{k}(\Lambda_{\bs{L}})$, in order to prove Proposition \ref{prop:firstminima}. We start off by presenting a few general results on lattices.

\begin{theorem}[Mahler-Weyl]
\label{theorem:Mahler-Weyl}
Let $\Lambda$ be a full-rank lattice in $\mb{R}^{n}$ with successive minima \\ $\lambda_{1},\dotsc,\lambda_{n}$. Then, there exists a basis $\bs{v}^{1},\dotsc,\bs{v}^{n}$ of $\Lambda$ such that
$$\lambda_{s}\leq|\bs{v}^{s}|_{2}\leq\max\left\{1,\frac{s}{2}\right\}\lambda_{s}$$
for $s=1,\dotsc,n$.
\end{theorem}

\noindent Theorem \ref{theorem:Mahler-Weyl} is a consequence of \cite[Chapter VIII, Lemma 1]{Cassels:AnIntrotoGeomofNumb} and \cite[Chapter V, Lemma 8]{Cassels:AnIntrotoGeomofNumb}.

Let
$$\mathfrak{I}\left(\bs{v}\right):=\left\{h\in\{1,\dotsc,\mc{M}+\mc{N}\}:v_{h}\neq 0\right\}$$
for $\bs{v}\in\mb{R}^{\mc{M}+\mc{N}}$. From Theorem \ref{theorem:Mahler-Weyl}, we deduce the following.

\begin{lem}
\label{lem:monotonicity}
Let $\Lambda$ be a full rank lattice in $\mb{R}^{n}$ and let $\lambda_{1},\dotsc,\lambda_{n}$ be its successive minima. Then, there exists a basis $\bs{v}^{1},\dotsc,\bs{v}^{n}$ of the lattice $\Lambda$ such that\vspace{2mm}
\begin{itemize}
\item[$i)$] $|\bs{v}^{s}|_{2}\ll_{n}\lambda_{s}$ for $s=1,\dotsc,n$;\vspace{2mm}
\item[$ii)$] $\mathfrak{I}\left(\bs{v}^{s}\right)\subseteq \mathfrak{I}\left(\bs{v}^{s+1}\right)$ for $s=1,\dotsc,n-1$.\vspace{2mm}
\end{itemize}
\end{lem}

\begin{proof}
By Theorem \ref{theorem:Mahler-Weyl}, there exists a basis $\bs{v}^{1},\dotsc,\bs{v}^{n}$ of $\Lambda$, such that $|\bs{v}^{s}|_{2}\ll_{n}\lambda_{s}$ for $s=1,\dotsc,n$. We set
\begin{equation}
\tilde{\bs{v}}^{s}:=\begin{cases}
\bs{v}^{1} & \quad\text{if }s=1 \\
\bs{v}^{s}+c_{s}\tilde{\bs{v}}^{s-1} & \quad\text{if }s>1
\end{cases},\nonumber
\end{equation}
where $c_{s}\in\mb{Z}$ are some coefficients yet to be chosen. Clearly, the vectors $\tilde{\bs{v}}^{1},\dotsc,\tilde{\bs{v}}^{n}$ form a basis for the lattice $\Lambda$ independently of the choice that we make for the coefficients $c_{s}$. We define these coefficients by recursion on $s$. Suppose that $\tilde{\bs{v}}^{s}$ are defined for all $s<\underline{s}$, where $1<\underline{s}\leq n$. To define $\tilde{\bs{v}}^{\underline{s}}$, we use the following procedure. First, we set $c_{\underline{s}}:=0$. If $\mathfrak{I}\left(\tilde{\bs{v}}^{\underline{s}-1}\right)\subset \mathfrak{I}\left(\bs{v}^{\underline{s}}\right)$, there is nothing to prove. If $\mathfrak{I}\left(\tilde{\bs{v}}^{\underline{s}-1}\right)\not\subset \mathfrak{I}\left(\bs{v}^{\underline{s}}\right)$, we change the value of the coefficient $c_{\underline{s}}$ to $1$. Then, if $\mathfrak{I}\left(\tilde{\bs{v}}^{\underline{s}-1}\right)\subset \mathfrak{I}\left(\bs{v}^{\underline{s}}+\tilde{\bs{v}}^{\underline{s}-1}\right)$, we have found a suitable value for the coefficient $c_{\underline{s}}$, otherwise, it means that the vectors $\tilde{\bs{v}}^{\underline{s}-1}$ and $\bs{v}^{\underline{s}}$ share some non-zero component of equal modulus but opposite sign. If this happens, we set $c_{\underline{s}}:=2$. Then, again, either we have found a suitable value for the coefficient $c_{\underline{s}}$, or $\mathfrak{I}\left(\tilde{\bs{v}}^{\underline{s}-1}\right)\not\subset \mathfrak{I}\left(\bs{v}^{\underline{s}}+2\tilde{\bs{v}}^{\underline{s}-1}\right)$. This implies that some of the non-zero components of the vector $\bs{v}^{\underline{s}}$ are twice the same components of the vector $\tilde{\bs{v}}^{\underline{s}-1}$, up to a change of sign. If this is the case, we set $c_{\underline{s}}:=3$, and so on. As soon as $\mathfrak{I}\left(\tilde{\bs{v}}^{\underline{s}-1}\right)\subset \mathfrak{I}\left(\bs{v}^{\underline{s}}+c_{\underline{s}}\tilde{\bs{v}}^{\underline{s}-1}\right)$, we fix the value of the coefficient $c_{\underline{s}}$. Each non-zero component of the vector $\bs{v}^{\underline{s}}$ can exclude at most one value for the coefficient $c_{\underline{s}}$. Hence, the process terminates in at most $n+1$ steps.

To conclude, we show by recursion on $s$ that $\left|\bs{v}^{s}\right|\ll_{n}\lambda_{s}$ for $s=1,\dotsc,n$. Suppose that this is true for all the indices less than a fixed index $\underline{s}>1$. Then, we have that
$$\left|\tilde{\bs{v}}^{\underline{s}}\right|_{2}=\left|\bs{v}^{\underline{s}-1}+c_{\underline{s}}\tilde{\bs{v}}^{\underline{s}-1}\right|_{2}\leq\left|\bs{v}^{\underline{s}}\right|_{2}+(n+1)\left|\tilde{\bs{v}}^{\underline{s}-1}\right|_{2}\ll_{n}\lambda_{\underline{s}}+(n+1)\lambda_{\underline{s}-1}\ll_{n}\lambda_{\underline{s}},$$
and this completes the proof. 
\end{proof}
Now, we move back to our original setting. We write $\bs{v}=(\bs{x},\bs{y})$, with $\bs{x}\in\mb{R}^{\mc{M}}$ and $\bs{y}\in\mb{R}^{\mc{N}}$, for vectors $\bs{v}\in\mb{R}^{\mc{M}+\mc{N}}$. We fix an index $k\in \mc{K}$ and we denote by $\lambda_{1},\dotsc,\lambda_{\mc{M}+\mc{\mc{N}}}$ the successive minima of the lattice $\omega_{2}\circ\omega_{1}\circ\hat{\varphi}_{k}(\Lambda_{\bs{L}})$.

By Lemma \ref{lem:monotonicity}, there exists a basis $\bs{v}^{1},\dotsc,\bs{v}^{\mc{M}+\mc{N}}$ of the lattice $\omega_{2}\circ\omega_{1}\circ\hat{\varphi}_{k}(\Lambda_{\bs{L}})$ such that
\begin{equation}
\label{eq:parti}
|\bs{v}^{s}|_{2}\ll_{\mc{M}+\mc{N}}\lambda_{s}\mbox{ for }s=1,\dotsc,\mc{M}+\mc{\mc{N}},
\end{equation}
and
\begin{equation}
\label{eq:partii}
\mathfrak{I}\left(\bs{v}^{s}\right)\subseteq \mathfrak{I}\left(\bs{v}^{s+1}\right)\mbox{ for }s=1,\dotsc,\mc{M}+\mc{N}-1.
\end{equation}
Moreover, by the definition of the maps $\varphi_{k}$ (Proposition \ref{thm:partition}), we can write
\begin{equation}
\label{eq:partiii}
\bs{v}^{s}=\left(\theta e^{a^{k}_{1}-c}\left(L_{1}\bs{q}^{s}+p^{s}_{1}\right),\dotsc,\theta e^{a^{k}_{\mc{M}}-c}\left(L_{\mc{M}}\bs{q}^{s}+p^{s}_{\mc{M}}\right),\theta^{-\frac{\mc{M}}{\mc{N}}}\frac{Q}{Q_{1}}q^{s}_{1},\dotsc,\theta^{-\frac{\mc{M}}{\mc{N}}}\frac{Q}{Q_{\mc{N}}}q^{s}_{\mc{N}}\right)
\end{equation}
for some fixed $\bs{p}^{s}\in\mb{Z}^{\mc{M}}$ and $\bs{q}^{s}\in\mb{Z}^{\mc{N}}$. Then, from (\ref{eq:parti}) and (\ref{eq:partiii}) we deduce that
\begin{multline}
\label{eq:length}
\lambda_{s}\gg_{\mc{M},\mc{N}}|\bs{v}^{s}|_{2}=\Bigg(\theta^{2} e^{2a^{k}_{1}}\left(L_{1}\bs{q}^{s}+p^{s}_{1}\right)^{2}+\dotsb+\theta^{2} e^{2a^{k}_{\mc{M}}}\left(L_{\mc{M}}\bs{q}^{s}+p^{s}_{\mc{M}}\right)^{2}+ \\
\left. +\ \theta^{-\frac{2\mc{M}}{\mc{N}}}\frac{Q^{2}}{Q_{1}^{2}}\left(q^{s}_{1}\right)^{2}+\dotsb+\theta^{-\frac{2\mc{M}}{\mc{N}}}\frac{Q^{2}}{Q^{2}_{\mc{N}}}\left(q^{s}_{\mc{N}}\right)^{2}\right)^{\frac{1}{2}}
\end{multline}
for $s=1,\dotsc,\mc{M}+\mc{N}$.

The following lemma shows that, without loss of generality, we can additionally assume that $\bs{q}^{s}\neq\bs{0}$ for all $s=1,\dotsc,\mc{M}+\mc{N}$.

\begin{lem}
\label{lem:qneq0}
Suppose that $\bs{q}^{s_{0}}=\bs{0}$ for some $1\leq s_{0}\leq \mc{M}+\mc{N}$. Then, we have that
\begin{equation}
\frac{\left(\varepsilon Q^{\mc{N}}\right)^{\frac{s}{\mc{M}+\mc{N}}}}{\lambda_{1}\dotsm\lambda_{s}}\ll_{\mc{M},\mc{N}}T^{s}\nonumber
\end{equation}
for $s=1,\dotsc,\mc{M}+\mc{N}$.
\end{lem}

\begin{proof}
By (\ref{eq:partii}), we have that $\bs{q}^{s}=\bs{0}$ for all $s\leq s_{0}$. It follows that $\bs{p}^{s}\neq\bs{0}$ for all $s\leq s_{0}$. Hence, by part $iia)$ of Proposition \ref{thm:partition}, we deduce that
$$\lambda_{s}\geq\lambda_{1}\geq\theta\min_{i}e^{a^{k}_{i}}\gg_{\mc{M},\mc{N}}\frac{\left(\varepsilon Q^{\mc{N}}\right)^{\frac{1}{\mc{M}+\mc{N}}}}{\varepsilon^{\frac{1}{\mc{M}}}}\frac{\varepsilon^{\frac{1}{\mc{M}}}}{T}=\frac{1}{T}\left(\varepsilon Q^{\mc{N}}\right)^{\frac{1}{\mc{M}+\mc{N}}}$$
for all $s=1,\dotsc,\mc{M}+\mc{N}$. The claim follows directly from this inequality.
\end{proof}

Now, given that $\bs{q}^{s}\neq\bs{0}$ for all $s=1,\dotsc,\mc{M}+\mc{N}$, we deduce that $v^{s}_{i}=\theta(L_{i}\bs{q}+p_{i})\neq 0$ for all $i=1,\dotsc,\mc{M}$, since the entries of the row $L_{i}$ of the matrix $\bs{L}$ along with the integer $1$ are linearly independent over $\mb{Q}$ for $i=1,\dotsc,\mc{M}$. Hence, in view of Lemma \ref{lem:qneq0}, we can make the assumption that
\begin{equation}
\label{eq:partiv}
\{1,\dotsc,\mc{M}\}\subsetneq\mathfrak{I}\left(\bs{v}^{s}\right)
\end{equation}
for all $s=1,\dotsc,\mc{M}+\mc{N}$.

To conclude this subsection, we show that, by conveniently permuting the variables $y_{1},\dotsc ,y_{\mc{N}}$, we can further assume that the sets $\mathfrak{I}(\bs{v}^{s})$ have a nice "triangular" structure.

\begin{lem}
\label{lem:monotonicity2}
There exists a permutation of the variables $y_{1},\dotsc,y_{\mc{N}}$ such that for all the indices $s,l\in\{1,\dotsc,\mc{M}+\mc{N}\}$ we have that
\begin{equation}
\label{eq:partv}
l\in\mathfrak{I}\left(\bs{v}^{s}\right)\Rightarrow\{1,\dotsc,l\}\subset\mathfrak{I}\left(\bs{v}^{s}\right).
\end{equation}
\end{lem}

\begin{proof}
By (\ref{eq:partiv}) we can assume that
$$\{1,\dotsb,\mc{M}\}\subsetneq \mathfrak{I}\left(\bs{v}^{s}\right)$$
for all $s=1,\dotsc,\mc{M}+\mc{N}$. Let us consider the sets
$$R_{s}:=\mathfrak{I}\left(\bs{v}^{s}\right)\setminus\left\{1,\dotsc,\mc{M}\right\}\neq\emptyset$$
for $s=1,\dotsc,\mc{M}+\mc{N}$, and let $R_{0}:=\emptyset$. Let also $r_{s}:=\# R_{s}$ for $s=0,\dotsc,\mc{M}+\mc{N}$. To define the required permutation, we send the variables in the set $\{y_{j}:\mc{M}+j\in R_{s}\setminus R_{s-1}\}$ to the variables in the set $\{y_{r_{s-1}+1},\dotsc,y_{r_{s}}\}$ for $s=1,\dotsc,\mc{M}+\mc{N}$, i.e., we reorder the variables so that the null components of each basis vector are the ones with higher indices. This simple procedure delivers the required result.
\end{proof}

Let $\mc{M}+1\leq s_{0}\leq\mc{M}+\mc{N}$. Since the vectors $\bs{v}^{1},\dotsc,\bs{v}^{s_{0}}$ cannot all lie in the subspace $\{y_{s_{0}-\mc{M}}=\dotsb =y_{\mc{N}}=0\}$ (which has dimension $s_{0}$), there must be an index $s\leq s_{0}$ such that one among the components $v^{s}_{s_{0}},\dotsc,v^{s}_{\mc{M}+\mc{N}}$ of the vector $\bs{v}^{s}$ is non-zero. By (\ref{eq:partii}), we can take $s=s_{0}$. Therefore, from (\ref{eq:partv}), we deduce that $v^{s_{0}}_{\mc{M}+1},\dotsc,v^{s_{0}}_{s_{0}}\neq 0$, and hence $\{1,\dotsc,s_{0}\}\subset \mathfrak{I}\left(\bs{v}^{s_{0}}\right)$. In view of this and of (\ref{eq:partiv}), we can assume that
\begin{equation}
\label{eq:partvi}
\{1,\dotsc,\max\{\mc{M}+1,s\}\}\subset\mathfrak{I}\left(\bs{v}^{s}\right)
\end{equation}
for all $s=1,\dotsc,\mc{M}+\mc{N}$.
Now, we have a sufficiently "nice" basis of the lattice $\omega_{2}\circ\omega_{1}\circ\hat{\varphi}_{k}(\Lambda_{\bs{L}})$ and we can proceed with the proof of Proposition \ref{prop:firstminima}.

\subsection{Proof}

Throughout this section we fix and index $k\in\mc{K}$ and a basis $\{\bs{v}^{1},\dotsc,\bs{v}^{\mc{M}+\mc{N}}\}$ of the lattice $\omega_{2}\circ\omega_{1}\circ\hat{\varphi}_{k}(\Lambda_{\bs{L}})$ satisfying (\ref{eq:parti}), (\ref{eq:partii}), (\ref{eq:partiii}), (\ref{eq:partv}), and (\ref{eq:partvi}). Note that (\ref{eq:partvi}) implies that $q^{s}_{1}\neq 0$ for all $s=1,\dotsc,\mc{M}+\mc{N}$.

As mentioned in the Introduction, the main issue with the proof is represented by the null components of the basis vectors $\bs{v}^{1},\dotsc,\bs{v}^{\mc{M}+\mc{N}}$. We start by  presenting an instructive "naive approach" to the problem, that shows how our estimates of the products $\lambda_{1}\dotsm\lambda_{s}$ depend on the number of non-null components in the vector $\bs{q}^{1}$. This approach was briefly discussed in the Introduction, and is expanded on in the following lemma. 

\begin{lem}
\label{lem:sleqM}
Let $1\leq h_{1}\leq\mc{N}$ be the largest index such that $q^{1}_{h_{1}}\neq 0$. Then, for all $\underline{s}=1,\dotsc,\mc{M}+\mc{N}$ we have that
\begin{equation}
\label{eq:sleqM}
\frac{\left(\varepsilon Q^{\mc{N}}\right)^{\frac{\underline{s}}{\mc{M}+\mc{N}}}}{\lambda_{1}\dotsm\lambda_{\underline{s}}}\ll_{\mc{M},\mc{N}}1+\left(\frac{\varepsilon Q^{\mc{N}}}{\phi\left(Q\right)}\right)^{\frac{\underline{s}}{\mc{M}+h_{1}}}.
\end{equation}
\end{lem}

\begin{proof}
From (\ref{eq:length}), we have that
\begin{multline}
\label{eq:lengthlem1}
\lambda_{1}\gg_{\mc{M},\mc{N}}\Bigg(\theta^{2} e^{2a^{k}_{1}}\left\|L_{1}\bs{q}^{1}\right\|^{2}+\dotsb+\theta^{2} e^{2a^{k}_{\mc{M}}}\left\|L_{\mc{M}}\bs{q}^{1}\right\|^{2}+ \\
\left. +\ \theta^{-\frac{2\mc{M}}{\mc{N}}}\frac{Q^{2}}{Q_{1}^{2}}\left(q^{1}_{1}\right)^{2}+\dotsb+\theta^{-\frac{2\mc{M}}{\mc{N}}}\frac{Q^{2}}{Q^{2}_{\mc{N}}}\left(q^{1}_{h_{1}}\right)^{2}\right)^{\frac{1}{2}}.
\end{multline}
We consider two different cases. Case $1$:
\begin{itemize}
\item for all $j=1,\dotsc,\mc{N}$, it holds $\left|q^{1}_{j}\right|\leq Q_{j}$.
\end{itemize}
In this case, we apply the standard arithmetic-geometric mean inequality to the right-hand side of (\ref{eq:lengthlem1}), finding
\begin{equation}
\label{eq:z}
\lambda_{1}\gg_{\mc{M},\mc{N}}\left(\theta^{\mc{M}\left(1-\frac{h_{1}}{\mc{N}}\right)}\prod_{i=1}^{\mc{M}}\left\|L_{i}\bs{q}^{1}\right\|\cdot
\prod_{j=1}^{h_{1}}\frac{Q}{Q_{j}}\left|q^{1}_{j}\right|\right)^{\frac{1}{\mc{M}+h_{1}}}.
\end{equation}
Since the matrix $\bs{L}$ is multiplicatively badly approximable, we deduce that
\begin{equation}
\label{eq:zz}
\prod_{i=1}^{\mc{M}}\left\|L_{i}\bs{q}^{1}\right\|\geq\frac{\phi\left(\left(\prod_{j=1}^{\mc{N}}\max\left\{1,\left|q^{1}_{j}\right|\right\}\right)^{\frac{1}{\mc{N}}}\right)}{\prod_{j=1}^{\mc{N}}\max\left\{1,\left|q^{1}_{j}\right|\right\}}\geq\frac{\phi\left(Q\right)}{\prod_{j=1}^{h_{1}}\left|q^{1}_{j}\right|}.
\end{equation}
Hence, on noting that
\begin{equation}
\label{eq:zzz}
\prod_{j=1}^{h_{1}}\frac{Q}{Q_{j}}\geq Q^{h_{1}-\mc{N}}\prod_{j=h_{1}+1}^{\mc{N}}Q_{j}\geq Q^{h_{1}-\mc{N}},
\end{equation}
and by substituting (\ref{eq:zz}) and (\ref{eq:zzz}) into (\ref{eq:z}), we obtain that
\begin{equation*}
\lambda_{1}\gg_{\mc{M},\mc{N}}\left(\theta^{\mc{M}\left(1-\frac{h_{1}}{\mc{N}}\right)}\phi(Q)Q^{h_{1}-\mc{N}}\right)^{\frac{1}{\mc{M}+h_{1}}}=\left(\varepsilon Q^{\mc{N}}\right)^{-\frac{1}{\mc{M}+h_{1}}+\frac{1}{\mc{M}+\mc{N}}}\phi(Q)^{\frac{1}{\mc{M}+h_{1}}},
\end{equation*}
whence
\begin{equation}
\label{eq:lengthlem12}
\frac{\left(\varepsilon Q^{\mc{N}}\right)^{\frac{\underline{s}}{\mc{M}+\mc{N}}}}{\lambda_{1}\dotsm\lambda_{\underline{s}}}\leq\frac{\left(\varepsilon Q^{\mc{N}}\right)^{\frac{\underline{s}}{\mc{M}+\mc{N}}}}{\lambda_{1}^{\underline{s}}}\ll_{\mc{M},\mc{N}}\left(\frac{\varepsilon Q^{\mc{N}}}{\phi\left(Q\right)}\right)^{\frac{\underline{s}}{\mc{M}+h_{1}}}.
\end{equation}
This completes the proof in case $1$.

Case $2$:
\begin{itemize}
\item there exist an index $1\leq j_{0}\leq h_{1}$ such that $\left|q^{1}_{j_{0}}\right|> Q_{j_{0}}$.
\end{itemize}
By ignoring all the terms but $\theta^{-2\mc{M}/\mc{N}}\left(Q/Q_{j_{0}}\right)^{2}(q^{1}_{j_{0}})^{2}$ in (\ref{eq:lengthlem1}), we obtain that
\begin{equation*}
\lambda_{1}\gg_{\mc{M},\mc{N}}\theta^{-\frac{\mc{M}}{\mc{N}}}Q=\left(\varepsilon Q^{\mc{N}}\right)^{\frac{1}{\mc{M}+\mc{N}}}.
\end{equation*}
Hence,
\begin{equation}
\label{eq:lengthlem13}
\frac{\left(\varepsilon Q^{\mc{N}}\right)^{\frac{\underline{s}}{\mc{M}+\mc{N}}}}{\lambda_{1}\dotsm\lambda_{\underline{s}}}\leq\frac{\left(\varepsilon Q^{\mc{N}}\right)^{\frac{\underline{s}}{\mc{M}+\mc{N}}}}{\lambda_{1}^{\underline{s}}}\ll_{\mc{M},\mc{N}}1.
\end{equation}
The result follows from (\ref{eq:lengthlem12}) and (\ref{eq:lengthlem13}).
\end{proof}

Now, to prove Proposition \ref{prop:firstminima}, we need the exponent of the ratio $\varepsilon Q^{\mc{N}}/\phi(Q)$ in (\ref{eq:sleqM}) to be less than or equal to $(\mc{M}+\mc{N}-1)/(\mc{M}+\mc{N})$. The fact that this exponent is strictly less than $1$ will be crucial to prove Theorem \ref{thm:cor2}. Lemma \ref{lem:sleqM} ensures that this holds true for $\underline{s}\leq\mc{M}$. However, for $\underline{s}\geq\mc{M}+1$, the result depends on the value of the index $h_{1}$. In particular, to deduce the desired estimate for
$$\frac{\left(\varepsilon Q^{\mc{N}}\right)^{\frac{\underline{s}}{\mc{M}+\mc{N}}}}{\lambda_{1}\dotsm\lambda_{\underline{s}}},$$
we need the number $\mc{M}+h_{1}$ to be at least $\underline{s}+1$, i.e., we need that $\{1,\dotsc,\underline{s}\}\subset\mathfrak{I}(\bs{v}^{1})$. This cannot be guaranteed, since the only information that we have with regards to the vector $\bs{q}^{1}$ is that $\bs{q}^{1}\neq\bs{0}$. 
In view of this, a slightly more sophisticated approach is required, and, in this approach, the higher successive minima $\lambda_{2},\dotsc,\lambda_{s}$ of the lattice $\omega_{2}\circ\omega_{1}\circ\hat{\varphi}_{k}(\Lambda_{\bs{L}})$ play a crucial role.

We will show that for a fixed index $\underline{s}$ it is not necessary to have $\{1,\dotsc,\underline{s}+1\}\subset\mathfrak{I}(\bs{v}^{1})$ to obtain (\ref{eq:firstminima}), but it suffices that $\{1,\dotsc,s+1\}\subset\mathfrak{I}(\bs{v}^{s})$ for all $s=1,\dotsc,\underline{s}$.

\begin{lem}
\label{lem:quickcase}
Let $\mc{M}\leq\underline{s}\leq \mc{M}+\mc{N}-1$ and suppose that for all $s=1,\dotsc,\underline{s}$ it holds $\{1,\dotsc,s+1\}\subseteq \mathfrak{I}\left(\bs{v}^{s}\right)$. Then, we have that
$$\frac{\left(\varepsilon Q^{\mc{N}}\right)^{\frac{\underline{s}}{\mc{M}+\mc{N}}}}{\lambda_{1}\dotsm\lambda_{\underline{s}}}\ll_{\mc{M},\mc{N}}1+\left(\frac{\varepsilon Q^{\mc{N}}}{\phi\left(Q\right)}\right)^{\frac{\underline{s}}{\underline{s}+1}}.$$
\end{lem}
\noindent The proof of Lemma \ref{lem:quickcase} is rather hefty and we postpone it to Section \ref{sec:quickcase}.

Let us fix an index $\underline{s}\in\{1,\dotsc,\mc{M}+\mc{N}\}$. Lemma \ref{lem:quickcase} is applicable whenever the condition $\{1,\dotsc,s+1\}\subset\mathfrak{I}\left(\bs{v}^{s}\right)$ holds for all $s=1,\dotsc,\underline{s}$. However, this is not always the case. Indeed, the only information that we have with regards to the basis vectors $\bs{v}^{1},\dotsc,\bs{v}^{\underline{s}}$ is that $\{1,\dotsc,\max\{\mc{M},s\}\}\subset\mathfrak{I}\left(\bs{v}^{s}\right)$ for $s=1,\dotsc,\underline{s}$. In particular, the condition $\{1,\dotsc,s+1\}\subset\mathfrak{I}\left(\bs{v}^{s}\right)$ never holds for $s=\mc{M}+\mc{N}$. Thus, we are left with one extra case to settle, i.e., the case when there exists an index $s\leq\underline{s}$ such that $\mathfrak{I}\left(\bs{v}^{s}\right)=\{1,\dotsc,s\}$. We deal with this case in the following lemma.

\begin{lem}
\label{lem:slowcase}
Let $\mc{M}+1\leq\underline{s}\leq \mc{M}+\mc{N}$ and assume that there exists an index $\mc{M}+1\leq s_{0}\leq\underline{s}$ such that $\mathfrak{I}\left(\bs{v}^{s_{0}}\right)=\{1,\dotsc,s_{0}\}$. Then, we have that
$$\frac{\left(\varepsilon Q^{\mc{N}}\right)^{\frac{\underline{s}}{\mc{M}+\mc{N}}}}{\lambda_{1}\dotsm\lambda_{\underline{s}}}\ll_{\mc{M},\mc{N}}\varepsilon Q^{\mc{N}}.$$
\end{lem}
\noindent The proof of Lemma \ref{lem:quickcase} can be found in Section \ref{sec:slowcase}.

Combining Lemmata \ref{lem:sleqM}, \ref{lem:quickcase}, and \ref{lem:slowcase}, we finally obtain a proof of Proposition \ref{prop:firstminima}. 

\section{Proof of Lemma \ref{lem:quickcase}}
\label{sec:quickcase}

This section is entirely devoted to proving Lemma \ref{lem:quickcase}. In the proof, we distinguish two cases. Case 1:
\begin{itemize}
\item for all $1\leq s\leq\underline{s}$ and all $1\leq j\leq\mc{N}$ it holds $\left|q^{s}_{j}\right|\leq Q_{j}$.
\end{itemize}
Let $1\leq h_{s}\leq \mc{N}$ be the largest index such that $q^{s}_{h_{s}}\neq 0$. By (\ref{eq:length}), we have that
\begin{align}
\label{eq:lengthlem3}
\left|\bs{v}^{s}\right|_{2} & \geq\Bigg(\theta^{2} e^{2a^{k}_{1}}\left\|L_{1}\bs{q}^{s}\right\|^{2}+\dotsb+\theta^{2} e^{2a^{k}_{\mc{M}}}\left\|L_{\mc{M}}\bs{q}^{s}\right\|^{2}+\nonumber \\
 & \hspace{5cm}+\theta^{-\frac{2\mc{M}}{\mc{N}}}\frac{Q^{2}}{Q_{1}^{2}}\left(q^{s}_{1}\right)^{2}+\dotsb+\frac{Q^{2}}{Q^{2}_{h_{s}}}\left(q^{s}_{h_{s}}\right)^{2}\Bigg)^{\frac{1}{2}},
\end{align} 
where $s$ ranges from $1$ to $\underline{s}$. Recall that, by (\ref{eq:partv}), the condition $q^{s}_{h_{s}}\neq 0$ implies that $q^{s}_{1},\dotsc,q^{s}_{h_{s}}\neq 0$ for all $s=1,\dotsc,\underline{s}$. Moreover, by (\ref{eq:partvi}) and the hypothesis, we have that $h_{s}\geq\max\{\mc{M}+1,s+1\}$.

To bound below the right-hand side of (\ref{eq:lengthlem3}), for each $s=1,\dotsc,\underline{s}$ we use a weighted arithmetic-geometric mean inequality. If the condition $\{1,\dotsc,\underline{s}+1\}\subset\mathfrak{I}(\bs{v}^{s})$ were true for $s=1,\dotsc,\underline{s}$, by applying the standard arithmetic-geometric mean inequality at each level $s$, the proof would be concluded. This is, roughly speaking, what we did in Lemma \ref{lem:sleqM}. Indeed, by construction, the condition $\{1,\dotsc,\underline{s}+1\}\subset\mathfrak{I}(\bs{v}^{s})$ for $s=1,\dotsc,\underline{s}$ is equivalent to the condition $\{1,\dotsc,\underline{s}+1\}\subset\mathfrak{I}(\bs{v}^{1})$. However, it is not always true that $\{1,\dotsc,\underline{s}+1\}\subset\mathfrak{I}(\bs{v}^{s})$ for all $s=1,\dotsc,\underline{s}$. To overcome this problem, we observe that if for some indices $\mc{M}+1\leq s<\underline{s}$ and $\mc{M}+1< l\leq\underline{s}+1$ the term $v^{s}_{l}$ is null, then, by the fact that $\{1,\dotsc,s+1\}\subset\mf{I}(\bs{v}^{s})$, there must be an index $s<s(l)\leq\underline{s}$ such that the $l$-th entry of the vector $\bs{v}^{s(l)}$ is non-null. The smallest such index $s(l)$ is $s(l)=l-1$. When applying a weighted arithmetic-geometric mean inequality at the level $l-1$, we will therefore assign a heavier weight to the term $v^{l-1}_{l}$, to compensate the (potential) absence of the $l$-th entry in the vectors $\bs{v}^{s}$ for $s<l-1$. By using this trick, we are "pretending" that $\{1,\dotsc,\underline{s}+1\}\subset\mathfrak{I}(\bs{v}^{s})$ for all $s=1,\dotsc,\underline{s}$.

Let us define the weights\footnote{If the real numbers $x_{1},\dotsc,x_{s+1}>0$ are assigned real weights $w_{1},\dotsc,w_{s+1}>0$, we set the sum of the weights to be $k:=w_{1}+\dotsb+w_{s+1}$, and we call weighted arithmetic-geometric mean inequality with weights $w_{1},\dotsc,w_{s+1}$ the following inequality:
$$k_{s}^{-1}(x_{1}+\dotsb+x_{s+1})\geq\left(x_{1}^{w_{1}}\dotsm x_{s+1}^{w_{s+1}}\right)^{1/k}.$$} more precisely. For $s\leq\mc{M}$ we assign weight $1$ to each term at the right-hand side of (\ref{eq:lengthlem3}). For $s\geq\mc{M}+1$ we assign weight $1$ to each term at the right-hand side of (\ref{eq:lengthlem3}) except for the term $v_{s+1}^{s}=\theta^{-\mc{M}/\mc{N}}Q/Q_{s+1-\mc{M}}\left|q^{s}_{s+1-\mc{M}}\right|$, to which we assign weight $1+(k_{s}-\mc{M}-h_{s})$, where $k_{s}$ is a rational number that still needs to be defined. Note that, with this definition, for each index $\mc{M}+1\leq s\leq\underline{s}$, the number $k_{s}$ denotes precisely the sum of all weights assigned at the level $s$. Now, we proceed to define the weight sums $k_{s}$ for all $s$. For $s=1,\dotsc,\mc{M}$ we set (according to our choice of weights)
\begin{equation}
\label{eq:exponentbasecase}
k_{s}:=\mc{M}+h_{s}.
\end{equation}
For $s\geq \mc{M}+1$, we define $k_{s}$ by recursion. Let $\bs{t}$ be the $(\mc{M}+\mc{N}-1)\times \mc{N}$ matrix whose entries are defined by
$$t_{sj}:=
\begin{cases}
1 & \quad\text{if }q^{s}_{j}\neq 0 \\
0 & \quad\text{if }q^{s}_{j}=0
\end{cases}.$$
In line with what we explained above, we require that the number $k_{s}$ satisfy the equation
\begin{equation}
\label{eq:fillingup}
\underbrace{\frac{1-t_{1(s+1-\mc{M})}}{k_{1}}+\dotsb+\frac{1-t_{(s-1)(s+1-\mc{M})}}{k_{s-1}}}_{\mbox{extra weight}}+\underbrace{\frac{1}{k_{s}}}_{\mbox{single weight}}=\underbrace{\frac{1+(k_{s}-\mc{M}-h_{s})}{k_{s}}}_{\mbox{total weight}}
\end{equation}
for $s=\mc{M}+1,\dotsc,\underline{s}$. Here, the left-hand side counts (with weights) how many times the $(s+1)$-th entry is null in the vectors $\bs{v}^{1},\dotsc,\bs{v}^{s-1}$ (extra weight) and adds this to the single weight (already divided by the weight sum $k_{s}$) of the term $v^{s}_{s+1}$ in the vector $\bs{v}^{s}$. The right-hand side represents the total weight assigned to the term $v^{s}_{s+1}$ at the level $s$ divided by the weight sum $k_{s}$. Imposing (\ref{eq:fillingup}) for a specific value of the index $s$ allows us to balance out the presence of the terms $Q/Q_{s+1-\mc{M}}$ in the product $\left|\bs{v}^{1}\right|_{2}\dotsm\left|\bs{v}^{\underline{s}}\right|_{2}$. This is crucial to obtain an upper bound only depending on the function $Q$ and not on the function $Q_{\max}$ in Proposition \ref{prop:firstminima} (see the Introduction for a rough explanation of why this condition is important).

From (\ref{eq:fillingup}) we deduce that
\begin{equation}
\label{eq:from40}
k_{s}:=(\mc{M}+h_{s})\left(1-\left(\frac{1-t_{1(s+1-\mc{M})}}{k_{1}}+\dotsb+\frac{1-t_{(s-1)(s+1-\mc{M})}}{k_{s-1}}\right)\right)^{-1}
\end{equation}
for $s=\mc{M}+1,\dotsc,\underline{s}$. Therefore, in order to show that the number $k_{s}$ is well defined, we have to prove that
\begin{equation}
1-\left(\frac{1-t_{1(s+1-\mc{M})}}{k_{1}}+\dotsb+\frac{1-t_{(s-1)(s+1-\mc{M})}}{k_{s-1}}\right)>0\nonumber
\end{equation}
for $s=\mc{M}+1,\dotsc,\underline{s}$. This always holds true, thanks to the following result.
\begin{lem}
\label{lem:exponentswelldef}
Let $k_{s}>0$ be real numbers such that (\ref{eq:exponentbasecase}) and (\ref{eq:fillingup}) hold for $s=1,\dotsc,\mc{M}+\mc{N}-1$. Let also
$$\alpha_{s}:=\frac{1}{k_{1}}+\dotsb+\frac{1}{k_{s}}$$
and
$$\alpha_{sj}:=\frac{t_{1j}}{k_{1}}+\dotsb+\frac{t_{sj}}{k_{s}}$$
for $s=1,\dotsc,\mc{M}+\mc{N}-1$ and $j=1,\dots,\mc{N}$. Then, under the hypotheses of Lemma \ref{lem:quickcase} and provided (\ref{eq:partv}) and (\ref{eq:partvi}) hold, we have that\vspace{2mm}
\begin{itemize}
\item[$i)$] $k_{s}\geq\mc{M}+h_{s}$ for $s=1,\dotsc,\mc{M}+\mc{N}-1$;\vspace{2mm}
\item[$ii)$] $\alpha_{s}(s+1)+\sum_{j>s+1-\mc{M}}\alpha_{sj}=s$ for $s=\mc{M},\dotsc,\mc{M}+\mc{N}-1$, whence $\alpha_{s}\leq s/(s+1)$.\vspace{2mm}
\end{itemize}
If $s+1-\mc{M}=\mc{N}$, the sum $\sum_{j>s+1-\mc{M}}\alpha_{sj}$ in part $ii)$ should be disregarded.
\end{lem}
\noindent We prove Lemma \ref{lem:exponentswelldef} in Section \ref{sec:exponentwelldef}.

As explained above, for $s=1,\dotsc,\underline{s}$ we apply a weighted arithmetic-geometric mean inequality to the right-hand side of (\ref{eq:lengthlem3}), assigning weight $1+(k_{s}-\mc{M}-h_{s})$ to the term $v^{s}_{s+1}=\theta^{-\mc{M}/\mc{N}}Q/Q_{s+1-\mc{M}}\left|q^{s}_{s+1-\mc{M}}\right|$ and weight $1$ to all other terms in the sum. In view of Proposition \ref{thm:partition} part $iia)$, by doing so we obtain
\begin{equation}
\left|\bs{v}^{s}\right|_{2}\gg_{\mc{M},\mc{N}}\left(\theta^{\mc{M}\left(1-\frac{k_{s}-\mc{M}}{\mc{N}}\right)}\prod_{i=1}^{\mc{M}}\left\|L_{i}\left(\bs{q}^{s}\right)\right\|\cdot
\prod_{j=1}^{h_{s}}\frac{Q}{Q_{j}}\left|q^{s}_{j}\right|\cdot\left(\frac{Q}{Q_{s+1-\mc{M}}}\left|q^{s}_{s+1-\mc{M}}\right|\right)^{k_{s}-\mc{M}-h_{s}}\right)^{\frac{1}{k_{s}}}\nonumber
\end{equation}
for $s=1,\dotsc,\underline{s}$. Now, since the matrix $\bs{L}$ is multiplicatively badly approximable, we have that
$$\prod_{i=1}^{\mc{M}}\left\|L_{i}\left(\bs{q}^{s}\right)\right\|\geq\frac{\phi\left(\left(\prod_{j=1}^{\mc{N}}\max\left\{1,\left|q^{s}_{j}\right|\right\}\right)^{\frac{1}{\mc{N}}}\right)}{\prod_{j=1}^{\mc{N}}\max\left\{1,\left|q^{s}_{j}\right|\right\}}\geq\frac{\phi\left(Q\right)}{\prod_{j=1}^{h_{s}}\left|q^{s}_{j}\right|}.$$
for all $s=1,\dots,\underline{s}$. Hence, we deduce that
\begin{equation}
\label{eq:lengthlem4'}
\left|\bs{v}^{s}\right|_{2}\gg_{\mc{M},\mc{N}}\left(\theta^{\mc{M}\left(1-\frac{k_{s}-\mc{M}}{\mc{N}}\right)}\cdot\phi\left(Q\right)\cdot\prod_{j=1}^{h_{s}}\frac{Q}{Q_{j}}\cdot\left(\frac{Q}{Q_{s+1-\mc{M}}}\right)^{k_{s}-\mc{M}-h_{s}}\right)^{\frac{1}{k_{s}}}
\end{equation}
for $s=1,\dotsc,\underline{s}$. We are then left to estimate the product $\left|\bs{v}^{1}\right|_{2}\dotsm\left|\bs{v}^{\underline{s}}\right|_{2}$. We do this by multiplying together the inequalities in (\ref{eq:lengthlem4'}). For simplicity, we compute the exponent of each factor (e.g., $\theta$, $\phi(Q)$, etc.) separately. We use the notation introduced in Lemma \ref{lem:exponentswelldef}. For the constant $\theta$ the exponent in the product of the inequalities in (\ref{eq:lengthlem4'}) is given by
\begin{equation}
\frac{\mc{M}}{k_{1}}\left(1-\frac{k_{1}-\mc{M}}{\mc{N}}\right)+\dotsb+\frac{\mc{M}}{k_{\underline{s}}}\left(1-\frac{k_{\underline{s}}-\mc{M}}{\mc{N}}\right)=\frac{\mc{M}(\mc{M}+\mc{N})}{\mc{N}}\alpha_{\underline{s}}-\frac{\mc{M}}{\mc{N}}\underline{s}.\nonumber
\end{equation}
For the function $\phi\left(Q\right)$ the exponent in the product of the inequalities in (\ref{eq:lengthlem4'}) is given by $\alpha_{\underline{s}}$. If $1\leq j\leq\underline{s}+1-\mc{M}$, by equation (\ref{eq:fillingup}), the ratios $Q/Q_{j}$ must all have the same exponent in the product of the inequalities in (\ref{eq:lengthlem4'}), i.e., $\alpha_{\underline{s}}$ (the computation is obvious for, e.g., $Q/Q_{1}$). On the other hand, if $j\geq\underline{s}+2-\mc{M}$ the exponent of the ratios $Q/Q_{j}$ is given by $\alpha_{\underline{s}j}$. Combining these considerations, we obtain that
\begin{equation}
\label{eq:lengthlem5}
\left|\bs{v}^{1}\right|_{2}\dotsm\left|\bs{v}^{\underline{s}}\right|_{2}\gg_{\mc{M},\mc{N}} \theta^{\frac{\mc{M}(\mc{M}+\mc{N})}{\mc{N}}\alpha_{\underline{s}}-\frac{\mc{M}}{\mc{N}}\underline{s}}\cdot\phi\left(Q\right)^{\alpha_{\underline{s}}}\cdot
\left(\prod_{j=1}^{\underline{s}+1-\mc{M}}\frac{Q}{Q_{j}}\right)^{\alpha_{\underline{s}}}\prod_{j=\underline{s}+2-\mc{M}}^{\mc{N}}\left(\frac{Q}{Q_{j}}\right)^{\alpha_{\underline{s}j}}.
\end{equation}
Now, we observe that
\begin{equation}
\label{eq:finaltheta}
\theta^{\frac{\mc{M}(\mc{M}+\mc{N})}{\mc{N}}\alpha_{\underline{s}}-\frac{\mc{M}}{\mc{N}}\underline{s}}=\left(\varepsilon^{-\frac{\mc{N}}{\mc{M}(\mc{M}+\mc{N})}}Q^{\frac{\mc{N}}{\mc{M}+\mc{N}}}\right)^{\frac{\mc{M}(\mc{M}+\mc{N})}{\mc{N}}\alpha_{\underline{s}}-\frac{\mc{M}}{\mc{N}}\underline{s}}=\varepsilon^{-\alpha_{\underline{s}}+\frac{\underline{s}}{\mc{M}+\mc{N}}}Q^{\mc{M}\alpha_{\underline{s}}-\frac{\mc{M}}{\mc{M}+\mc{N}}\underline{s}}.
\end{equation}
Moreover, by using the fact that $Q$ is the geometric mean of the parameters $Q_{1},\dotsc,Q_{\mc{N}}$, we have that
\begin{multline}
\label{eq:expratios3}
\left(\prod_{j=1}^{\underline{s}+1-\mc{M}}\frac{Q}{Q_{j}}\right)^{\alpha_{\underline{s}}}\prod_{j=\underline{s}+2-\mc{M}}^{\mc{N}}\left(\frac{Q}{Q_{j}}\right)^{\alpha_{\underline{s}j}}= \\
=\prod_{j=\underline{s}+2-\mc{M}}^{\mc{N}}\left(\frac{Q}{Q_{j}}\right)^{\alpha_{\underline{s}j}-\alpha_{\underline{s}}}\geq Q^{-\alpha_{\underline{s}}(\mc{N}+\mc{M}-\underline{s}-1)+\sum_{j=\underline{s}+2-\mc{M}}^{\mc{N}}\alpha_{\underline{s}j}},
\end{multline}
where the lower bound is obtained by trivially setting $Q_{j}=1$ for $\underline{s}+2-\mc{M}\leq j\leq \mc{N}$. By Lemma \ref{lem:exponentswelldef}, it holds
$$\alpha_{\underline{s}}(\underline{s}+1)+\sum_{j=\underline{s}+2-\mc{M}}^{\mc{N}}\alpha_{\underline{s}j}=\underline{s},$$
therefore, Equation (\ref{eq:expratios3}) implies that
\begin{equation}
\label{eq:nnew}
\left(\prod_{j=1}^{\underline{s}+1-\mc{M}}\frac{Q}{Q_{j}}\right)^{\alpha_{\underline{s}}}\prod_{j=\underline{s}+2-\mc{M}}^{\mc{N}}\left(\frac{Q}{Q_{j}}\right)^{\alpha_{\underline{s}j}}\geq Q^{-\alpha_{\underline{s}}(\mc{N}+\mc{M})+\underline{s}}.
\end{equation}
Combining (\ref{eq:finaltheta}) and (\ref{eq:nnew}) with (\ref{eq:lengthlem5}), we finally obtain that
$$\left|\bs{v}^{1}\right|_{2}\dotsm\left|\bs{v}^{\underline{s}}\right|_{2}\gg_{\mc{M},\mc{N}}\varepsilon^{-\alpha_{\underline{s}}+\frac{\underline{s}}{\mc{M}+\mc{N}}}Q^{\mc{M}\alpha_{\underline{s}}-\frac{\mc{M}}{\mc{M}+\mc{N}}\underline{s}-\alpha_{\underline{s}}(\mc{N}+\mc{M})+\underline{s}}\phi(Q)^{\alpha_{\underline{s}}}=\left(\varepsilon Q^{\mc{N}}\right)^{-\alpha_{\underline{s}}+\frac{\underline{s}}{\mc{M}+\mc{N}}}\phi(Q)^{\alpha_{\underline{s}}}.$$
This, in turn, yields
$$\frac{\left(\varepsilon Q^{\mc{N}}\right)^{\frac{\underline{s}}{\mc{M}+\mc{N}}}}{\lambda_{1}\dotsm\lambda_{\underline{s}}}\ll_{\mc{M},\mc{N}}\left(\frac{\varepsilon Q^{\mc{N}}}{\phi\left(Q\right)}\right)^{\alpha_{\underline{s}}}\ll_{\mc{M},\mc{N}}1+\left(\frac{\varepsilon Q^{\mc{N}}}{\phi\left(Q\right)}\right)^{\frac{\underline{s}}{\underline{s}+1}},$$
where the last inequality is due to Lemma \ref{lem:exponentswelldef} part $ii)$. This completes the proof in case $1$.

Case 2:
\begin{itemize}
\item there exist some indices $1\leq s_{0}\leq\underline{s}$ and $1\leq j_{0}\leq \mc{N}$ such that $\left|q^{s_{0}}_{j_{0}}\right|> Q_{j_{0}}$.
\end{itemize}
Suppose that $s_{0}$ is the least index such that $\left|q^{s_{0}}_{j_{0}}\right|> Q_{j_{0}}$ for some $1\leq j_{0}\leq \mc{N}$. Then, by ignoring all the terms but $\theta^{-2\mc{M}/\mc{N}}\left(Q/Q_{j_{0}}\right)^{2}(q^{s_{0}}_{j_{0}})^{2}$ in (\ref{eq:lengthlem3}), we find that
$$\lambda_{s}\gg_{\mc{M},\mc{N}}\theta^{-\frac{\mc{M}}{\mc{N}}}Q=\varepsilon^{\frac{1}{\mc{M}+\mc{N}}}Q^{\frac{\mc{N}}{\mc{M}+\mc{N}}}$$
for all $s_{0}\leq s\leq\underline{s}$. Hence, we can write
$$\frac{\left(\varepsilon Q^{\mc{N}}\right)^{\frac{\underline{s}}{\mc{M}+\mc{N}}}}{\lambda_{1}\dotsm\lambda_{\underline{s}}}\ll_{\mc{M},\mc{N}}\frac{\left(\varepsilon Q^{\mc{N}}\right)^{\frac{s_{0}-1}{\mc{M}+\mc{N}}}}{\lambda_{1}\dotsm\lambda_{s_{0}-1}}\frac{\left(\varepsilon Q^{\mc{N}}\right)^{\frac{\underline{s}-s_{0}+1}{\mc{M}+\mc{N}}}}{\left(\varepsilon Q^{\mc{N}}\right)^{\frac{\underline{s}-s_{0}+1}{\mc{M}+\mc{N}}}}=\frac{\left(\varepsilon Q^{\mc{N}}\right)^{\frac{s_{0}-1}{\mc{M}+\mc{N}}}}{\lambda_{1}\dotsm\lambda_{s_{0}-1}}.$$
To conclude the proof, it suffices to apply case 1 (if $s_{0}\geq\mc{M}+1$) or Lemma \ref{lem:sleqM} (if $s_{0}\leq\mc{M}$).

\section{Proof of Lemma \ref{lem:slowcase}}
\label{sec:slowcase}

Let $\bs{B}_{k}\in\mb{R}^{(\mc{M}+\mc{N})\times(\mc{M}+\mc{N})}$ be the matrix that represents the linear transformation $\omega_{2}\circ\omega_{1}\circ\hat{\varphi}_{k}$ in the canonical basis. Then, we have that
$$\omega_{1}\circ\omega_{2}\circ\hat{\varphi}_{k}(\Lambda_{\bs{L}})=\bs{B}_{k}\bs{A}_{\bs{L}}\mb{Z}^{\mc{M}+\mc{N}}.$$
Let $\left(\bs{B}_{k}\bs{A}_{\bs{L}}\right)_{s_{0}}$ be the $s_{0}\times s_{0}$ submatrix of $\bs{B}_{k}\bs{A}_{\bs{L}}$ formed by the first $s_{0}$ rows and the first $s_{0}$ columns. Let also
$$ \bs{C}_{k}:=\left(\begin{array}{@{}c|c@{}}
   \left(\bs{B}_{k}\bs{A}_{\bs{L}}\right)_{s_{0_{\phantom{a}}}}\!\!\! & \bs{0} \\\hline
   \accentset{\phantom{A}}{\bs{0}} & \accentset{\phantom{A}}{\bs{0}}
  \end{array}\right)\in\mb{R}^{(\mc{M}+\mc{N})\times(\mc{M}+\mc{N})},$$
and
$$\Lambda_{s_{0}}:=\bs{C}_{k}\mb{Z}^{\mc{M}+\mc{N}}.$$
The lattice $\Lambda_{s_{0}}$ has rank $s_{0}$, and its projection onto the first $s_{0}$ coordinates has covolume
\begin{equation}
\label{eq:slowcase0}
\Det\left(\Lambda_{s_{0}}\right)=\theta^{\mc{M}\left(1-\frac{s_{0}-\mc{M}}{\mc{N}}\right)}\prod_{j=1}^{s_{0}-\mc{M}}\frac{Q}{Q_{j}}.
\end{equation}
By the hypothesis, we also have that $\bs{v}_{1},\dotsc,\bs{v}_{s_{0}}\in\Lambda_{s_{0}}.$
Since the vectors $\bs{v}_{1},\dotsc,\bs{v}_{s_{0}}$ are linearly independent, by Minkowsky's Theorem (see \cite[Chapter VIII, Theorem I]{Cassels:AnIntrotoGeomofNumb}), we deduce that
\begin{equation}
\label{eq:slowcase1}
\lambda_{1}\dotsm\lambda_{s_{0}}\gg_{\mc{M},\mc{N}}\prod_{s=1}^{s_{0}}\left|\bs{v}^{s}\right|_{2}\geq\lambda_{1}(\Lambda_{s_{0}})\dotsm\lambda_{s_{0}}(\Lambda_{s_{0}})\gg_{\mc{M},\mc{N}}\Det\left(\Lambda_{s_{0}}\right),
\end{equation}
where $\lambda_{s}(\Lambda_{s_{0}})$ ($s=1,\dotsc,s_{0}$) are the successive minima of the projection of the lattice $\Lambda_{s_{0}}$ onto the first $s_{0}$ coordinates. This gives us an estimate of the product $\lambda_{1}\dotsm\lambda_{s_{0}}$.

We are left to estimate the product $\lambda_{s_{0}+1}\dotsm\lambda_{\underline{s}}$. By (\ref{eq:partvi}), we have that $\{1,\dotsc,s\}\subseteq \mathfrak{I}\left(\bs{v}^{s}\right)$ for $s=\mc{M}+1,\dotsc,\mc{M}+\mc{N}$. Hence, by using (\ref{eq:length}), we can write
\begin{equation}
\label{eq:slowcase3}
\lambda_{s}\gg_{\mc{M},\mc{N}}|\bs{v}^{s}|_{2}\geq\theta^{-\frac{\mc{M}}{\mc{N}}}\frac{Q}{Q_{s-\mc{M}}}\left|q^{s}_{s-\mc{M}}\right|\geq\theta^{-\frac{\mc{M}}{\mc{N}}}\frac{Q}{Q_{s-\mc{M}}}
\end{equation}
for $s=s_{0}+1,\dotsc,\underline{s}$, where we ignored all the terms but $\theta^{-\frac{\mc{M}}{\mc{N}}}\frac{Q}{Q_{s-\mc{M}}}\left|q^{s}_{s-\mc{M}}\right|$. Then, by (\ref{eq:slowcase0}), (\ref{eq:slowcase1}), and (\ref{eq:slowcase3}), we deduce that
\begin{align}
\lambda_{1}\dotsm\lambda_{\underline{s}} & \gg_{\mc{M},\mc{N}}\Det\left(\Lambda_{s_{0}}\right)\prod_{s=s_{0}+1}^{\underline{s}}\theta^{-\frac{\mc{M}}{\mc{N}}}\frac{Q}{Q_{s-\mc{M}}}=\nonumber \\
 & =\theta^{\mc{M}\left(1-\frac{s_{0}-\mc{M}}{\mc{N}}\right)}\prod_{j=1}^{s_{0}-\mc{M}}\frac{Q}{Q_{j}}\prod_{j=s_{0}-\mc{M}+1}^{\underline{s}-\mc{M}}\theta^{-\frac{\mc{M}}{\mc{N}}}\frac{Q}{Q_{j}}=\nonumber \\
 & =\theta^{\mc{M}\left(1-\frac{\underline{s}-\mc{M}}{\mc{N}}\right)}\frac{1}{Q^{\mc{M}+\mc{N}-\underline{s}}}\prod_{j=\underline{s}+1-\mc{M}}^{\mc{N}}Q_{j}\geq\nonumber \\
 & \geq\theta^{\mc{M}\left(1-\frac{\underline{s}-\mc{M}}{\mc{N}}\right)}Q^{\underline{s}-(\mc{M}+\mc{N})}=\left(\varepsilon Q^{\mc{N}}\right)^{\frac{\underline{s}}{\mc{M}+\mc{N}}-1},\nonumber
\end{align}
where $\prod_{s=s_{0}+1}^{\underline{s}}\theta^{-\frac{\mc{M}}{\mc{N}}}Q/Q_{s-\mc{M}}=1$ if $s_{0}=\underline{s}$.
Hence, we find that
\begin{equation}
\frac{\left(\varepsilon Q^{\mc{N}}\right)^{\frac{\underline{s}}{\mc{M}+\mc{N}}}}{\lambda_{1}\dotsm\lambda_{\underline{s}}}\ll_{\mc{M},\mc{N}}\frac{\left(\varepsilon Q^{\mc{N}}\right)^{\frac{\underline{s}}{\mc{M}+\mc{N}}}}{\left(\varepsilon Q^{\mc{N}}\right)^{\frac{\underline{s}}{\mc{M}+\mc{N}}-1}}=\varepsilon Q^{\mc{N}}.\nonumber
\end{equation}

\section{Proof of Lemma \ref{lem:exponentswelldef}}
\label{sec:exponentwelldef}

Throughout this section, we denote (once again) by $1\leq h_{s}\leq \mc{N}$ the largest non-zero index such that $q^{s}_{h_{s}}\neq 0$ for $s=1,\dotsc,\mc{M}+\mc{N}-1$.

If $\underline{s}<\mc{M}$, by definition, we have that 
$$k_{\underline{s}}:=\mc{M}+h_{\underline{s}},$$
and part $i)$ holds true. If $\underline{s}\geq \mc{M}$, we simultaneously prove parts $i)$ and $ii)$ by recursion on $\underline{s}$. First, we let $\underline{s}=\mc{M}$. Since
$$k_{\mc{M}}:=\mc{M}+h_{\mc{M}},$$
part $i)$ holds true. Further, we observe that
\begin{equation}
\label{eq:lem3.71}
(\mc{M}+1)\sum_{s=1}^{\mc{M}}\frac{1}{\mc{M}+h_{s}}+\sum_{j=2}^{\mc{N}}\sum_{s=1}^{\mc{M}}\frac{t_{sj}}{\mc{M}+h_{s}}=\sum_{s=1}^{\mc{M}}\frac{\mc{M}+1+\sum_{j=2}^{\mc{N}}t_{sj}}{\mc{M}+h_{s}},
\end{equation}
and, since $q_{1}^{s}\neq 0$ for $s=1,\dotsc,\mc{M}$, we also have that
\begin{equation}
\label{eq:lem3.72}
1+\sum_{j=2}^{\mc{N}}t_{sj}=h_{s}.
\end{equation}
Hence, equations (\ref{eq:lem3.71}) and (\ref{eq:lem3.72}) imply part $ii)$. Now, let us fix $\mc{M}<\underline{s}\leq \mc{M}+\mc{N}-1$, and let us suppose that both parts $i)$ and $ii)$ hold for all the indices $s$ such that $\mc{M}\leq s<\underline{s}$. By (\ref{eq:from40}), we have that
\begin{equation}
k_{\underline{s}}:=(h_{\underline{s}}+\mc{M})\left(1-\left(\frac{1-t_{1(\underline{s}+1-\mc{M})}}{k_{1}}+\dotsb+\frac{1-t_{(\underline{s}-1)(\underline{s}+1-\mc{M})}}{k_{\underline{s}-1}}\right)\right)^{-1}.\nonumber
\end{equation}
Hence, to prove part $i)$ for $\underline{s}$, it suffices to show that
$$0\leq\frac{1-t_{1(\underline{s}+1-\mc{M})}}{k_{1}}+\dotsb+\frac{1-t_{(\underline{s}-1)(\underline{s}+1-\mc{M})}}{k_{\underline{s}-1}}<1.$$
Since either $t_{sj}=0$ or $t_{sj}=1$ for all $s$ and $j$, the recursive hypothesis for part $i)$ ($s<\underline{s}$) implies that
\begin{equation}
\label{eq:0}
\frac{1-t_{(\underline{s}+1-\mc{M})1}}{k_{1}}+\dotsb+\frac{1-t_{(\underline{s}-1)(\underline{s}+1-\mc{M})}}{k_{\underline{s}-1}}\geq0.
\end{equation}
To prove the other inequality, we observe that
\begin{equation}
\label{eq:0.5}
\frac{1-t_{(\underline{s}+1-\mc{M})1}}{k_{1}}+\dotsb+\frac{1-t_{(\underline{s}-1)(\underline{s}+1-\mc{M})}}{k_{\underline{s}-1}}\leq\frac{1}{k_{1}}+\dotsb+\frac{1}{k_{\underline{s}-1}}=\alpha_{\underline{s}-1}.
\end{equation}
Since the recursive hypothesis for part $ii)$ ($s=\underline{s}-1$) implies that
$$\alpha_{\underline{s}-1}\underline{s}+\sum_{j>\underline{s}-\mc{M}}\alpha_{(\underline{s}-1)j}=\underline{s}-1,$$
and, since $\alpha_{(\underline{s}-1)j}\geq 0$ for all $j$, we deduce that
\begin{equation}
\label{eq:1}
\alpha_{\underline{s}-1}\leq\frac{\underline{s}-1}{\underline{s}}<1.
\end{equation}
Hence, combining (\ref{eq:0}), (\ref{eq:0.5}), and (\ref{eq:1}), we obtain that
$$0\leq\frac{1-t_{(\underline{s}+1-\mc{M})1}}{k_{1}}+\dotsb+\frac{1-t_{(\underline{s}-1)(\underline{s}+1-\mc{M})}}{k_{\underline{s}-1}}<1.$$
We are now left to prove part $ii)$ for $s=\underline{s}$. We start by observing that
\begin{equation}
\label{eq:finallemma1}
\alpha_{\underline{s}}(\underline{s}+1)+\sum_{j>\underline{s}+1-\mc{M}}\alpha_{\underline{s}j}=\alpha_{\underline{s}-1}\underline{s}+\alpha_{\underline{s}-1}+\frac{1}{k_{\underline{s}}}(\underline{s}+1)+\sum_{j>\underline{s}+1-\mc{M}}\alpha_{(\underline{s}-1)j}+\sum_{j>\underline{s}+1-\mc{M}}\frac{t_{\underline{s}j}}{k_{\underline{s}}}.
\end{equation}
We claim that
\begin{equation}
\label{eq:finallemma2}
\alpha_{\underline{s}-1}+\frac{1}{k_{\underline{s}}}(\underline{s}+1)+\sum_{j>\underline{s}+1-\mc{M}}\frac{t_{\underline{s}j}}{k_{\underline{s}}}=1+\alpha_{(\underline{s}-1)(\underline{s}+1-\mc{M})}.
\end{equation}
This concludes the proof, since (\ref{eq:finallemma1}), (\ref{eq:finallemma2}), and the recursive hypothesis for part $ii)$ ($s=\underline{s}-1$) imply that
$$\alpha_{\underline{s}}(\underline{s}+1)+\sum_{j>\underline{s}+1-\mc{M}}\alpha_{\underline{s}j}=\alpha_{\underline{s}-1}\underline{s}+\alpha_{(\underline{s}-1)(\underline{s}+1-\mc{M})}+\sum_{j>\underline{s}+1-\mc{M}}\alpha_{(\underline{s}-1)j}+1=\underline{s}-1+1.$$

Now, we prove (\ref{eq:finallemma2}). By assumption, $t_{\underline{s}j}=1$ for $1\leq j\leq\underline{s}+1-\mc{M}$, hence, we have that
$$\frac{1}{k_{\underline{s}}}(\underline{s}+1)=\frac{1}{k_{\underline{s}}}\left(\mc{M}+\sum_{j=1}^{\underline{s}+1-\mc{M}}t_{\underline{s}j}\right)$$
whence
\begin{equation}
\label{eq:explanation0}
\alpha_{\underline{s}-1}+\frac{1}{k_{\underline{s}}}(\underline{s}+1)+\sum_{j>\underline{s}+1-\mc{M}}\frac{t_{\underline{s}j}}{k_{\underline{s}}}=\alpha_{\underline{s}-1}+\frac{1}{k_{\underline{s}}}\left(\mc{M}+\sum_{j=1}^{\mc{N}}t_{\underline{s}j}\right)=\alpha_{\underline{s}-1}+\frac{\mc{M}+h_{\underline{s}}}{k_{\underline{s}}}.
\end{equation}
Further, we observe that (\ref{eq:fillingup}) for $s=\underline{s}$ can be rewritten as
\begin{equation}
\label{eq:explanation}
\alpha_{\underline{s}-1}-\alpha_{(\underline{s}-1)(\underline{s}+1-\mc{M})}=1-\frac{\mc{M}+h_{\underline{s}}}{k_{\underline{s}}}.
\end{equation}
Thus, combining (\ref{eq:explanation0}) and (\ref{eq:explanation}), we obtain that
$$\alpha_{\underline{s}-1}+\frac{1}{k_{\underline{s}}}(\underline{s}+1)+\sum_{j>\underline{s}+1-\mc{M}}\frac{t_{\underline{s}j}}{k_{\underline{s}}}=1+\alpha_{(\underline{s}-1)(\underline{s}+1-\mc{M})},$$
which proves (\ref{eq:finallemma2}).

\section{Proof of Theorem \ref{thm:cor2}}
\label{sec:proofofthm}

We start by noticing that
\begin{align}
 & \sum_{\substack{\bs{q}\in\prod_{j=1}^{\mc{N}}[-Q_{j},Q_{j}]\\ \cap\ \mb{Z}^{\mc{N}}\setminus\{\bs{0}\}}}\prod_{i=1}^{\mc{M}}\|L_{i}\bs{q}\|^{-1}\nonumber \\
 & \leq\sum_{k=0}^{\infty}2^{k+1}\#\left\{\bs{q}\in\prod_{j=1}^{\mc{N}}[-Q_{j},Q_{j}]\cap\mb{Z}^{\mc{N}}\setminus\{\bs{0}\}:2^{-k-1}\leq\prod_{i=1}^{\mc{M}}\|L_{i}\bs{q}\|<2^{-k}\right\} \nonumber \\
 & \leq\sum_{k=0}^{\infty}2^{k+1}\#\left\{\bs{q}\in\prod_{j=1}^{\mc{N}}[-Q_{j},Q_{j}]\cap\mb{Z}^{\mc{N}}\setminus\{\bs{0}\}:\prod_{i=1}^{\mc{M}}\|L_{i}\bs{q}\|<2^{-k}\right\}.\nonumber
\end{align}

From this, (\ref{eq:intersection}), and Lemma \ref{lem:emptycase} with $\varepsilon=2^{-k}$, we deduce that

\begin{align}
 \label{eq:cor2eq1} 
 \sum_{\substack{\bs{q}\in\prod_{j=1}^{\mc{N}}[-Q_{j},Q_{j}]\\ \cap\ \mb{Z}^{\mc{N}}\setminus\{\bs{0}\}}}\prod_{i=1}^{\mc{M}}\|L_{i}\bs{q}\|^{-1} & \leq\sum_{k=0}^{\infty}2^{k+1}\# M\left(\bs{L},2^{-k},\frac{1}{2},\bs{Q}\right)\nonumber \\
 & =\sum_{k=0}^{\left\lfloor\log_{2}\left(\frac{Q^{\mc{N}}}{\phi(Q)}\right)\right\rfloor}2^{k+1}\# M\left(\bs{L},2^{-k},\frac{1}{2},\bs{Q}\right).
\end{align}

We use Proposition \ref{prop:cor1} to estimate the right-hand side of (\ref{eq:cor2eq1}). We need $T^{\mc{M}}/\varepsilon\geq e^{\mc{M}}$, i.e., $2^{k-\mc{M}}\geq e^{\mc{M}}$. Hence, we split the sum in (\ref{eq:cor2eq1}) into two parts, one for $2^{k-\mc{M}}<e^{\mc{M}}$ and one for $2^{k-\mc{M}}\geq e^{\mc{M}}$. We find that
\begin{align}
 & \sum_{\substack{\bs{q}\in\prod_{j=1}^{\mc{N}}[-Q_{j},Q_{j}]\\ \cap\ \mb{Z}^{\mc{N}}\setminus\{\bs{0}\}}}\prod_{i=1}^{\mc{M}}\|L_{i}\bs{q}\|^{-1}\leq\sum_{k=0}^{\left\lfloor\mc{M}\left(1+1/\log 2\right)\right\rfloor}2^{k+1}\# M\left(\bs{L},2^{-k},\frac{1}{2},\bs{Q}\right)\nonumber \\
 & +\sum_{k=\left\lceil\mc{M}\left(1+1/\log 2\right)\right\rceil}^{\left\lfloor\log_{2}\left(\frac{Q^{\mc{N}}}{\phi(Q)}\right)\right\rfloor}2^{k+1}\# M\left(\bs{L},2^{-k},\frac{1}{2},\bs{Q}\right)\nonumber \\
 & \ll_{\mc{M},\mc{N}}Q^{\mc{N}}+\sum_{k=\left\lceil\mc{M}\left(1+1/\log 2\right)\right\rceil}^{\left\lfloor\log_{2}\left(\frac{Q^{\mc{N}}}{\phi(Q)}\right)\right\rfloor}2^{k+1}(k-\mc{M})^{\mc{M}-1}\left(2^{-k}Q^{\mc{N}}+\left(\frac{2^{-k}Q^{\mc{N}}}{\phi(Q)}\right)^{\frac{\mc{M}+\mc{N}-1}{\mc{M}+\mc{N}}}\right)\label{eq:cor2eq2.1}  \\
 & \ll_{\mc{M},\mc{N}}\sum_{k=0}^{\left\lfloor\log_{2}\left(\frac{Q^{\mc{N}}}{\phi(Q)}\right)\right\rfloor}k^{\mc{M}-1}\left(Q^{\mc{N}}+2^{\frac{k}{\mc{M}+\mc{N}}}\left(\frac{Q^{\mc{N}}}{\phi(Q)}\right)^{\frac{\mc{M}+\mc{N}-1}{\mc{M}+\mc{N}}}\right)\label{eq:cor2eq2.2},
\end{align}
where in (\ref{eq:cor2eq2.1}) we estimate $\# M\left(\bs{L},2^{-k},1/2,\bs{Q}\right)$ with $Q^{\mc{N}}$ for $k\leq \left\lfloor\mc{M}\left(1+1/\log 2\right)\right\rfloor$. Note that $Q\geq 2$ ensures that (\ref{eq:cor2eq2.1})$\Rightarrow$(\ref{eq:cor2eq2.2}). The required result follows from (\ref{eq:cor2eq2.2}) combined with the trivial estimates $\sum_{k=0}^{K}k^{\mc{M}-1}\ll_{\mc{M}} K^{\mc{M}}$ and $\sum_{k=0}^{K}k^{\mc{M}-1}2^{\frac{k}{\mc{M}+\mc{N}}}\ll_{\mc{M},\mc{N}} K^{\mc{M}-1}2^{\frac{K}{\mc{M}+\mc{N}}}$.

\addcontentsline{toc}{section}{\bibname}
\bibliographystyle{plain}
\bibliography{Bibliography}

\end{document}